\documentclass[12pt,a4paper,reqno]{amsart}
\allowdisplaybreaks

\usepackage[a4paper,left=3cm,right=3cm]{geometry}
\usepackage[foot]{amsaddr}
\usepackage{amsthm}
\usepackage{amssymb}
\usepackage{amsmath}
\usepackage{mathtools}
\usepackage{amsfonts}
\usepackage{bbm}
\usepackage{xcolor}
\usepackage{enumitem}
\usepackage{mathabx}
\usepackage{tikz-cd}
\usepackage{caption}
\usepackage{ulem}	
\usepackage[abs]{overpic}

\usepackage{tikz}
\usetikzlibrary{intersections,fpu}

\usepackage{hyperref}
\hypersetup{
  colorlinks   = true, 
  urlcolor     = blue, 
  linkcolor    = blue, 
  citecolor   = blue, 
}

\newcommand{\rleq}{\trianglelefteq}
\newcommand{\rgeq}{\trianglerighteq}

\newtheorem{proposition}{Proposition}
\newtheorem{theorem}[proposition]{Theorem}
\newtheorem{lemma}[proposition]{Lemma}
\newtheorem{corollary}[proposition]{Corollary}

\theoremstyle{definition}
\newtheorem{definition}[proposition]{Definition}

\newtheorem{remark}[proposition]{Remark}

\definecolor{lighter}{RGB}{10.0,146.0,35.0}
\definecolor{darker}{RGB}{2.0,100.0,20.0}
\definecolor{sininen}{RGB}{89.5,101.8,238.1}
\definecolor{punainen}{RGB}{227.0,12.0,12.0}

\newcommand{\darker}[1]{\textcolor{darker}{#1}}
\newcommand{\punainen}[1]{\textcolor{punainen}{#1}}
\newcommand{\sininen}[1]{\textcolor{sininen}{#1}}

\newcommand{\mc}[1]{\mathcal{#1}}

\newcommand{\mf}[1]{\mathfrak{#1}}
\newcommand{\mb}[1]{\mathbb{#1}}

\newcommand{\N}{\mathbb N}
\newcommand{\R}{\mathbb R}

\newcommand{\T}{\mathbb T}

\def\<{\langle}
\def\>{\rangle}

\newcommand{\tj}{\vartheta}

\let\originalleft\left
\let\originalright\right
\renewcommand{\left}{\mathopen{}\mathclose\bgroup\originalleft}
\renewcommand{\right}{\aftergroup\egroup\originalright}

\makeatletter
\newcommand*{\saved@uline}{}
\let\saved@uline\uline
\newcommand*{\mathuline}{%
  \mathpalette{\math@uline\saved@uline}%
}
\newcommand*{\math@uline}[3]{%
  \mbox{#1{$#2#3\m@th$}}%
}

\begin{document}

\title{Multivariate majorization\\ of continuous statistical experiments}

\author{Erkka Haapasalo}
\address{Centre for Quantum  Technologies,  National University of Singapore}
\email{cqteth@nus.edu.sg}

\begin{abstract}
We derive sufficient and almost necessary conditions for large sample and catalytic majorization between finite statistical experiments over standard Borel sample spaces. This work generalizes previous results, on one hand, in the bivariate case and, on the other hand, in the multivariate discrete (or, rather, finite) case, i.e.,\ matrix majorization. We derive multivariate generalizations of the bivariate R\'{e}nyi relative entropies and show that inequalities involving these multivariate R\'{e}nyi divergences characterize large-sample and catalytic majorization of finite statistical experiments. As our methods are real-algebraic in nature, this work demonstrates that large deviation techniques are not the only option available to derive conditions for large sample majorization even in the case of more general sample spaces of the experiments. We also show that all general multivariate divergences, i.e.,\ multivariate extensive and monotone maps of finite statistical experiments, can be expressed through barycentres over the set of multivariate R\'{e}nyi divergences. We also show that we may characterize the optimal conversion rate of a statistical experiment into another using the multivariate R\'{e}nyi divergences.
\end{abstract}

\maketitle

\section{Introduction}

{\it Statistical experiments} are basic notions in statistics and information theory. They are collections of information from a phenomenon, e.g.,\ weather, physical experiment, or a medical study. Information on a sample space $(X,\mc A)$, a measurable space, comes in the form of probability measures $p:\mc A\to[0,1]$, so statistical experiments are collections $P=\{p^\tj\}_{\tj\in\Theta}$ of probability measures on some sample space and some fixed collection $\Theta$ of settings. It is natural to say that a statistical experiment $P=\{p^\tj\}_{\tj\in\Theta}$ is more informative than another statistical experiment $Q=\{q^\tj\}_{\tj\in\Theta}$ if there is a stochastic operator $T$ (modeled as a Markov kernel) such that $Tp^\tj=q^\tj$ for all $\tj\in\Theta$ \cite{Blackwell51,Blackwell53,Torgersen91}. Note that the experiment $P$ contains at least the information of $Q$ as we may produce tha data of $Q$ out of $P$ using a single data processing operation. In this situation, we say that $P$ {\it majorizes} $Q$ and denote $P\succeq Q$.

Having many copies of the data of the statistical experiment or an auxiliary or `catalytic' experiment in our disposal makes majorization easier to characterize. To this end, given statistical experiments $P=\{p^\tj\}_{\tj\in\Theta}$ and $Q=\{q^\tj\}_{\tj\in\Theta}$, we denote
$$
P\boxtimes Q:=\big\{p^\tj\otimes q^\tj\big\}_{\tj\in\Theta}
$$
where $p\otimes q$ is the product measure of measures $p$ and $q$ (defined on the product $\sigma$-algebra); we use this notation instead of the more standard $p\times q$ for internal cohesion throughout this work. We say that $P$ majorizes $Q$ {\it in large samples} if $P^{\boxtimes n}\succeq Q^{\boxtimes n}$ whenever $n\in\N$ is large enough. Furthermore, we say that $P$ {\it catalytically majorizes} $Q$ if there exists a catalytic experiment $R$ such that $P\boxtimes Q\succeq Q\boxtimes R$. Large-sample majorization implies catalytic majorization \cite{Duan2005} but the converse is not true \cite{Feng_et_al_2006}. The exact difference between these two modes of majorization is still quite unknown in general though.

In this work we focus on {\it finite} statistical experiments \cite{Heyer1982}. This means that the fixed set $\Theta$ of settings is finite. In the special case $\Theta=\{1,2\}$, majorization is also called {\it relative majorization}. Relative majorization and conditions for large-sample relative majorization were exhaustively studied by Mu {\it et al.}\ in \cite{Mu_et_al_2021}. The methods employed therein are based on large deviation techniques. However, these techniques, although useful in the relative case, do not readily lend themselves for the study of {\it multivariate majorization} where $\Theta=\{1,\ldots,d\}$ for a general $d\geq 2$. Building upon the real-algebraic machinery of \cite{FritzI,FritzII} and the work done in the finite-outcome case in \cite{Farooq_et_al_2024,Verhagen_et_al_2024}, we are now able to fully generalize the results of \cite{Mu_et_al_2021} to the multivariate case. Majorization in the case of finite statistical experiments with finite sample spaces studied in \cite{Farooq_et_al_2024,Verhagen_et_al_2024} is also called {\it matrix majorization} because, in this case, we may view an experiment $P=\big(p^{(1)},\ldots,p^{(d)}\big)$ with finite probability distributions $p^{(k)}:X\to[0,1]$ ($X$ being a finite set) as the matrix $\big(p^{(k)}(x)\big)_{k,x}$ which we multiply with a stochastic matrix modeling a finite Markov kernel \cite{dahl99}. We must point out however, that we have to place a certain boundedness condition on the finite statistical experiments $\{p^{(k)}\}_{k=1}^d$ that we study: all the measures $p^{(k)}$ within the same experiment must belong to the same measure class (i.e.,\ they are mutually dominating) and all the Radon-Nikod\'{y}m derivatives $dp^{(k)}/dp^{(\ell)}$, $k,\ell=1,\ldots,d$, must be essentially bounded (w.r.t.\ the shared measure class of the experiment). However, this same assumption is also made in the relative or bipartite case in \cite{Mu_et_al_2021}. We wish to highlight by this the fact that we obtain more general (multivariate) results with similar restrictions using real-algebraic methods than when using the more standard large-deviation methods.

We find that $P$ majorizes $Q$ in large samples (as well as catalytically) if $D(P)>D(Q)$ for {\it multivariate R\'{e}nyi divergences} $D$ generalizing the bivariate R\'{e}nyi relative entropies \cite{Renyi61}. The set of the multivariate R\'{e}nyi divergences consists of the {\it temperate quantities} $D_{\mathuline{\alpha}}$, together with their certain pointwise limits, defined by a parameter vector $\mathuline{\alpha}=(\alpha_1,\ldots,\alpha_d)\in\R^d$, with $\alpha_1+\cdots+\alpha_d=1$ and either $0\leq\alpha_\ell<1$ for all $\ell\in\{1,\ldots,d\}$ or $\alpha_k>1$ for one $k$ and $\alpha_\ell\leq0$ whenever $\ell\neq k$, through
\begin{align*}
D_{\mathuline{\alpha}}(P)&=\frac{1}{\max_{1\leq\ell\leq d}\alpha_\ell-1}\log{\mb E_{p^{(d)}}\left[\left(\frac{dp^{(1)}}{dp^{(d)}}\right)^{\alpha_1}\cdots\left(\frac{dp^{(d-1)}}{dp^{(d)}}\right)^{\alpha_{d-1}}\right]}\\
&=\frac{1}{\max_{1\leq\ell\leq d}\alpha_\ell-1}\log{\int_X\frac{dp^{(1)}}{dp^{(d)}}(x)^{\alpha_1}\cdots\frac{dp^{(d-1)}}{dp^{(d)}}(x)^{\alpha_{d-1}}\,dp^{(d)}(x)}
\end{align*}
for all statistical experiments $P=\big(p^{(1)},\ldots,p^{(d)}\big)$ over any (standard Borel) sample space $(X,\mc A)$ satisfying the boundedness condition detailed above. Note that $\alpha_d$ does not make an explicit appearance in the above formula, but we still include it in the definition to make the description of the allowed parameter vectors easier to state and understand. It might appear that the index $k=d$ has a special role in the above formula. Note however that since all the probability measures in the same statistical experiment are required to belong to the same measure class, we may transition from Radon-Nikod\'{y}m derivatives w.r.t.\ the final measure in the above formula to densities w.r.t.\ any chosen $p^{(k)}$ in which case the parameter $\alpha_k$ would not make an explicit appearance.

This is a full generalization of Theorem 1 of \cite{Mu_et_al_2021} which states that similar inequalities involving the traditional R\'{e}nyi relative entropies give a sufficient condition for large-sample relative majorization. This characterization of large-sample relative majorization was further used in \cite{Mu_et_al_2021} to derive an exhaustive characterization of all relative entropies. Relative entropies are to be understood as extensive (tensor-additive) maps that satisfy the data processing inequality. Using our multivariate generalization of the characterization of large-sample majorization, we have a complete characterization of all the multivariate divergences using the same functional-analytic methods that were used in \cite{Mu_et_al_2021}. In our case, a multivariate divergence $D$ is a real-valued function on the class of statistical experiments characterized by extensivity ($D(P\boxtimes Q)=D(P)+D(Q)$ for all $P$ and $Q$) and the data processing inequality ($P\succeq Q$ $\Rightarrow$ $D(P)\geq D(Q)$).

We also address the problem of characterizing the optimal rate of transforming a finite statistical experiment $P$ into another experiment $Q$. We will show that, if $P$ contains no repeating probability measure, this optimal rate can be expressed as
$$
r(P\to Q)=\inf_{\mathuline{\alpha}\in A}\frac{D_{\mathuline{\alpha}}(P)}{D_{\mathuline{\alpha}}(Q)}
$$
where the infimum is taken over the allowed region $A$ of parameter vectors $\mathuline{\alpha}$.

We apply a {\it Vergleichsstellensatz} from \cite{FritzII} which characterizes large-sample and catalytic ordering in particular preordered semirings using specific monotone homomorphisms and derivations of the semiring. We demonstrate that these maps give rise to the multivariate R\'{e}nyi divergences discussed above. We are thus able to derive sufficient (and almost necessary) conditions for large-sample and catalytic multivariate majorization in the form of inequalities between the values of these divergences on the two finite statistical experiments we are studying. The real-algebraic methods employed in this work were pioneered by Strassen \cite{strassen1986,strassen1987} when deriving the first sub-cubic matrix multiplication algorithm. This initiated Strassen's theory of asymptotic spectra \cite{strassen1991,strassen1998} resulting in his {\it Positivstellensatz}. The {\it Vergleichsstellens\"{a}tze} of \cite{FritzI,FritzII} can be seen as generalizations of the {\it Positivstellensatz}. Other generalizations can be found in \cite{Vrana2022}. A review of the different modern applications of preordered semirings and asymptotic spectra can be found in \cite{WigZuid}.

One of our main messages is that the real-algebraic methods are equally applicable in the continuous setting as in the finite-outcome setting. Recently a paper \cite{balsubramani2026} was published with closely related and in many cases identical results to the ones derived in this paper. In particular, Theorem 5.1 of \cite{balsubramani2026} is identical to our Theorem \ref{thm:barycentre} characterizing all extensive multivariate divergences as barycentres over the multivariate R\'{e}nyi divergences. We wish to highlight that this same result can be reached using various different techniques. Our approach is to use the real-algebraic machinery of \cite{FritzII} to derive large-sample majorization conditions and, using them and the results of \cite{haapasalo2025}, we arrive at the same characterization as that in Theorem 5.1 of \cite{balsubramani2026}.

This paper is organized as follows: We start with introducing basic notations used throughout in this work in Section \ref{sec:Background} followed with a quick recap of concepts related with preordered semirings in Subsection \ref{subsec:preorderedsemirings}. In this subsection, we also state the {\it Vergleichsstellensatz}, Theorem \ref{thm:Vergleichsstellensatz}, derived in \cite{FritzII} characterizing large-sample and catalytic order through inequalities over the test spectrum of the preordered semiring, i.e.,\ the set of specific monotone semiring homomorphisms and derivations over the preordered semiring. In section \ref{sec:MajSemiring}, we define the majorization semiring relevant for multivariate majorization and characterize the test spectrum of this preordered semiring building upon and generalizing the finite results of \cite{Farooq_et_al_2024,Verhagen_et_al_2024}. We use the {\it Vergleichsstellensatz} together with the characterization of the test spectrum to state the sufficient and almost necessary conditions for large-sample and catalytic majorization in Theorem \ref{thm:LSCatExact}. Using this large-sample result together with the same functional analytic techniques as those used in \cite{Mu_et_al_2021}, we also derive Theorem \ref{thm:barycentre} stating that all multivariate divergences are barycentres over the test spectrum, i.e.,\ the set of multivariate R\'{e}nyi divergences. In Subsection \ref{subsec:rates}, we derive the form for the optimal achievable rate of transforming a finite statistical experiment into another. We end with conclusions and future prospects in Section \ref{sec:concl}.

\section{Background}\label{sec:Background}

We use the following basic notations throughout this work:
\begin{itemize}
\item $\N=\{1,2,3,\ldots\}$
\item $\R_{>0}$: the set of all strictly positive real numbers.
\item $\R_+$: The set of non-negative real numbers seen as a preordered semiring (see the definition of this notion later) with the natural addition, multiplication, and (total) order $\geq$.
\item $\R_+^{\rm op}$: as a set the same as $\R_+$ above but with the opposite order $a\geq^{\rm op}b$ $\Leftrightarrow$ $b\geq a$. $\R_+$ and $\R_+^{\rm op}$ together are often called {\it temperate non-negative reals}.
\item $\T\R_+$: otherwise the same as $\R_+$ above but with the tropical sum $a+'b=\max\{a,b\}$.
\item $\T\R_+^{\rm op}$: the same as $\T\R_+$ above but with the opposite order. $\T\R_+$ and $\T\R_+^{\rm op}$ together are often called as {\it tropical non-negative reals}.
\end{itemize}

\subsection{Preordered semirings and a Vergleichsstellensatz}\label{subsec:preorderedsemirings}

Preordered semirings are essentially otherwise like commutative rings with the exception that the elements of a (preordered) semiring typically do not have additive inverses and there is a preorder which respects the algebraic structure. To formalize this, a preordered semiring is an ordered tuple $(S,0,1,+,\cdot,\rgeq)$ where $S\neq\emptyset$, $(S,0,+)$ and $(S,1,\cdot)$ are commutative semigroups where the multiplication distributes over the addition, and $\rgeq$ is a preorder (a reflexive and transitive binary relation) on $S$ such that
$$
x\rgeq y\ \Rightarrow\ \left\{\begin{array}{l}
x+a\rgeq y+a,\\
xa\rgeq ya,
\end{array}\right.\quad\forall a\in S.
$$
For a preordered semiring $S$ and $x,y\in S$, we denote $x\sim y$ whenever there are $z_1,\ldots,z_n\in S$ such that
$$
x\rgeq z_1\rleq z_2\rgeq\cdots\rleq z_n\rgeq y.
$$
We say that a preordered semiring $S$ is of {\it polynomial growth} when there is $u\in S$ such that
$$
x\rgeq y\ \Rightarrow\ \exists k\in\N:\ yu^k\rgeq x.
$$
Such an element $u$ is called a {\it power universal}. A preordered semiring $S$ is {\it zerosumfree} if
$$
x+y=0\ \Rightarrow\ x=0=y.
$$
We say that a preordered semiring $S$ is a {\it preordered semidomain} if
$$
xy=0\ \Rightarrow\ x=0\ {\rm or}\ y=0
$$
and
$$
0\rgeq x\rgeq 0\ \Rightarrow\ x=0.
$$

Given two preordered semirings $(S,0_S,1_S,+,\cdot,\rgeq_S)$ and $(T,0_T,1_T,+,\cdot,\rgeq_T)$, we say that a map $\Phi:S\to T$ is a {\it monotone homomorphism} if
\begin{itemize}
\item $\Phi(0_S)=0_T$, $\Phi(1_S)=1_T$,
\item $\Phi(x+y)=\Phi(x)+\Phi(y)$ for all $x,y\in S$,
\item $\Phi(xy)=\Phi(x)\Phi(y)$ for all $x,y\in S$, and
\item $x\rgeq_S y$ $\Rightarrow$ $\Phi(x)\rgeq_T\Phi(y)$.
\end{itemize}
We say that a monotone homomorphism $\Phi:S\to T$ is {\it degenerate} if $x\rgeq y$ implies $\Phi(x)=\Phi(y)$. Otherwise $\Phi$ is {\it nondegenerate}. We are primarily interested in monotone homomorphisms $\Phi:S\to\mb K$ where $\mb K\in\{\R_+,\R_+^{\rm op},\T\R_+,\T\R_+^{\rm op}\}$. Given a monotone homomorphism $\Phi:S\to\R_+$, we say that a map $\Delta:S\to\R$ is a {\it monotone derivation at $\Phi$} if $x\rgeq y$ implies $\Delta(x)\geq\Delta(y)$, $\Delta(x+y)=\Delta(x)+\Delta(y)$ for all $x,y\in S$, and
$$
\Delta(xy)=\Delta(x)\Phi(y)+\Phi(x)\Delta(y)\qquad{\rm (Leibniz\ rule)}
$$
for all $x,y\in S$. We are only interested in monotone derivations at degenerate homomorphisms $\Phi$ in which case we may view $\Phi$ as a monotone homomorphism of $S$ to $\R_+$ or, equivalently, $\R_+^{\rm op}$.

\begin{definition}\label{def:deg}
We say that a preordered semiring $S$ is {\it of degeneracy $d$} for some $d\in\N$ if there is a surjective homomorphism $\|\cdot\|:S\to\R_{>0}^d\cup\{(0,\ldots,0)\}$ with trivial kernel such that
$$
a\rgeq b\ \Rightarrow\ \|a\|=\|b\|,\quad \|a\|=\|b\|\ \Rightarrow\ a\sim b.
$$
In this situation, we denote the component homomorphisms of $\|\cdot\|$ by $\|\cdot\|_{(k)}$ for $k=1,\ldots,d$.
\end{definition}

Let us briefly note that, if $S$ is of degeneracy $d$, we have $x\sim y$ if and only if $\|x\|=\|y\|$. Naturally, $\|x\|=\|y\|$ $\Rightarrow$ $x\sim y$. For the converse, assume that $x\sim y$ and $z_1,\ldots,z_n\in S$ are such that $x\rgeq z_1\rleq z_2\rgeq\cdots\rleq z_n\rgeq y$. The first inequality in the chain implies $\|x\|=\|z_1\|$, the second implies $\|z_1\|=\|z_2\|$, and so on until $\|z_{n-1}\|=\|z_n\|$ and $\|z_n\|=\|y\|$. Thus,
$$
\|x\|=\|z_1\|=\|z_2\|=\cdots=\|z_{n-1}\|=\|z_n\|=\|y\|.
$$

\begin{definition}\label{def:TestSpectrum}
Let $S$ be a preordered semiring of polynomial growth and of degeneracy $d$ where we fix a power universal $u$ and the vector-valued homomorphism $\|\cdot\|$ of Definition \ref{def:TestSpectrum}. We denote the set of all nondegenerate monotone homomorphisms $\Phi:S\to\mb K$ with $\mb K\in\{\R_+,\R_+^{\rm op},\T\R_+,\T\R_+^{\rm op}\}$ by $\Sigma(S,\mb K)$ and the set of all the monotone derivations $\Delta:S\to\R$ at $\|\cdot\|_{(k)}$ ($k=1,\ldots,d$) with $\Delta(u)=1$ by $\mf D^k(S)$. For $\mb K\in\{\R_+,\R_+^{\rm op},\T\R_+,\T\R_+^{\rm op}\}$, we define
$$
\mf D(S,\mb K):=\left\{\frac{\log{\Phi(\cdot)}}{\log{\Phi(u)}}\,\middle|\,\Phi\in\Sigma(S,\mb K)\right\}
$$
as a set of maps $S\setminus\{0\}\to\R$. For $k\in\{1,\ldots,d\}$, we define
$$
\mf D_k(S):=\left\{S\setminus\{0\}\ni x\mapsto\frac{\Delta(x)}{\|x\|_{(k)}}\,\middle|\,\Delta\in\mf D^k(S)\right\}.
$$
Finally, we define the {\it test spectrum} through
$$
\hat{\mf D}(S):=\mf D(S,\R_+)\cup\mf D(S,\R_+^{\rm op})\cup\mf D(S,\T\R_+)\cup\mf D(S,\T\R_+^{\rm op})\cup\mf D_1(S)\cup\cdots\cup\mf D_d(S).
$$
\end{definition}

In \cite{FritzII}, the test spectrum was defined as the union of the sets $\Sigma(S,\mb K)$ with $\mb K\in\{\R_+,\R_+^{\rm op},\T\R_+,\T\R_+^{\rm op}\}$ and $\mf D^k(S)$ for $k=1,\ldots,d$. However, it is useful for our future discussion to define the test spectrum as the set $\hat{\mf D}(S)$ defined in Definition \ref{def:TestSpectrum}. The following central result ({\it Vergleichsstellensatz}) for this work is Theorem 8.6 in \cite{FritzII}:

\begin{theorem}\label{thm:Vergleichsstellensatz}
Let $S$ be a zerosumfree preordered semidomain of polynomial growth and of degeneracy $d$ for some $d\in\N$. We fix a power universal $u\in S$ and the vector-valued homomorphism $\|\cdot\|$ of Definition \ref{def:deg} to define the test spectrum $\hat{\mf D}(S)$. Suppose that $x,y\in S\setminus\{0\}$ are such that $\|x\|=\|y\|$. If
\begin{equation}\label{eq:FritzConditions}
\Delta(x)>\Delta(y)\qquad\forall\Delta\in\hat{\mf D}(S),
\end{equation}
then
\begin{itemize}
\item[(a)] $x^nu^k\rgeq y^nu^k$ for some $k\in\N$ when $n\in\N$ is sufficiently large and
\item[(b)] $xz\rgeq yz$ for some $z\in S\setminus\{0\}$ which can be chosen according to
$$
z=u^k\sum_{\ell=0}^n x^\ell y^{n-\ell}
$$
for some $k\in\N$ and $n\in\N$ sufficiently large.
\end{itemize}
If $x$ is a power universal, we may omit the appearance of $u^k$ in items (a) and (b) above. Conversely, if condition (a) or (b) above holds, then the inequalities in \eqref{eq:FritzConditions} hold non-strictly.
\end{theorem}

Before going on, let us note that any power universal $u\in S$ has $\|u\|=(1,\ldots,1)$ as long as $\|\cdot\|$ satisfies the conditions of Definition \ref{def:deg}. This is naturally equivalent with $u\sim 1$. To see this, assume that $x,y\in S$ are such that $x\rleq y$, so that there is $k\in\N$ such that $xu^k\rgeq y$. Since $x\rleq y$, we have $\|x\|=\|y\|$ and, since $xu^k\rgeq y$, we have $\|x\|\cdot\|u\|^k=\|y\|=\|x\|$ where we have used the first observation in the final equality. From this we immediately see that $\|u\|=(1,\ldots,1)$.

\subsection{Sample spaces and stochastic maps}

We are concentrating on standard Borel measurable spaces in this work. Recall that a measurable space $(X,\mc A)$ (i.e.,\ a non-empty set $X$ with a $\sigma$-algebra $\mc A$ of subsets of $X$) is standard Borel if it is $\sigma$-isomorphic with $\big(E,\mc B(E)\big)$ where $E$ is a Borel measurable subset of a Polish space $P$ (i.e.,\ a separable completely metrizable topological space) and $\mc B(E)$ is the restriction of the Borel $\sigma$-algebra $\mc B(P)$ onto $E$; $\mc B(E):=\{E\cap B\,|\,B\in\mc B(P)\}$. Most of our work is completely applicable to more general measurable spaces. However, stochastic maps between standard Borel spaces have a particularly nice form as they are discribed by Markov kernels which we subsequently discuss. Another reason is a bit more subtle: We subsequently have to discuss a proper class of tuples of finite positive measures over {\it all} standard Borel measurable spaces. However, we shall build a semiring out of this class by identifying $\sigma$-isomorphic tuples with each other. As any standard Borel measurable space is $\sigma$-isomorphic with a subspace of $\big(\R,\mc B(\R)\big)$, the semiring is a proper set.

In the sequel, we need to discuss certain sample spaces composed of two component sample spaces in order to define the binary operations (sum and multiplication) that we need in order to apply Theorem \ref{thm:Vergleichsstellensatz}. Let $(X,\mc A)$ and $(Y,\mc B)$ be standard Borel measurable spaces. We denote by $\mc A\otimes\mc B$ the coarsest $\sigma$-algebra on $X\times Y$ which contains the product sets $A\times B$, $A\in\mc A$, $B\in\mc B$. Given measures $\mu:\mc A\to\R$ and $\nu:\mc B\to\R$, we denote their product measure by $\mu\otimes\nu$. We also denote by $X\sqcup Y$ the disjoint union of $X$ and $Y$ which we equip with the $\sigma$-algebra $\mc A\oplus\mc B$ which is the coarsest $\sigma$-algebra on $X\sqcup Y$ containing the disjoint unions $A\sqcup B$, $A\in\mc A$ and $B\in\mc B$. Given measures $\mu:\mc A\to\R$ and $\nu:\mc B\to\R$, we define the measure $\mu\oplus\nu:\mc A\oplus\mc B\to\R$ through $(\mu\oplus\nu)(C)=\mu(C\cap X)+\nu(C\cap Y)$ for all $C\in\mc A\oplus\mc B$. Naturally, also $(X\times Y,\mc A\otimes\mc B)$ and $(X\sqcup Y,\mc A\oplus\mc B)$ are standard Borel measurable spaces.

Let $(X,\mc A)$ and $(Y,\mc B)$ be standard Borel measurable spaces. We say that $T(\cdot|\cdot):\mc B\times X\to[0,1]$ is a {\it Markov kernel} if $T(\cdot|x):\mc B\to[0,1]$ is a probability measure for all $x\in X$ and $T(B|\cdot):X\to[0,1]$ is $\mc A$-measurable for all $B\in\mc B$. Given a finite measure $\mu:\mc A\to\R$, we may define a new finite measure $T\mu:\mc B\to\R$, $(T\mu)(B)=\int_X T(B|x)\,d\mu(x)$ for all $B\in\mc B$. Clearly $(T\mu)(Y)=\mu(X)$. Consequently, when $\mu$ is a probability measure, so is $T\mu$. More general stochastic (affine) maps mapping probability measures on $(X,\mc A)$ into probability measures on $(Y,\mc B)$ are described by weak Markov kernels \cite{JePuVi2008} which define bimeasures on $\mc A\times\mc B$. However, when the input and output spaces are standard Borel, the bimeasures extend into measures on $(X\times Y,\mc A\otimes\mc B)$ \cite[Lemma 4.2.1]{QTOS}, so that we may replace weak Markov kernels with ordinary Markov kernels as defined above.

\section{The majorization semiring}\label{sec:MajSemiring}

From nor on, we fix $d\in\N$. Given a standard Borel measurable space $(X,\mc A)$, we say that $P=\big(p^{(1)},\ldots,p^{(d)}\big)$ is a {\it regular finite statistical experiment} (of size $d$) if $p^{(k)}:\mc A\to[0,1]$ are probability measures within the same measure class (i.e.,\ $p^{(k)}\ll p^{(\ell)}$ for $k,\ell=1,\ldots,d$) such that there is a real number $R>0$ with
\begin{equation}\label{eq:finite}
p^{(\ell)}\mathrm{-ess}\,\sup\frac{dp^{(k)}}{dp^{(\ell)}}\leq R,\qquad k,\ell\in\{1,\ldots,d\}.
\end{equation}
It immediately follows that, for all $k,\ell\in\{1,\ldots,d\}$, $dp^{(k)}/dp^{(\ell)}$ is also $p^{(\ell)}$-essentially bounded away from zero by $1/R>0$. Note that the attribute `regular' refers to the condition in \eqref{eq:finite} and `finite' refers to $P$ being a finite ($d<\infty$) collection of probability measures. Given finite positive measures $q^{(k)}:\mc A\to\R_+$, we say that $Q=\big(q^{(1)},\ldots,q^{(d)}\big)$ is a {\it regular $d$-tuple (over $(X,\mc A)$)} if $q^{(k)}=0$ for $k=1,\ldots,d$ or $q^{(k)}\neq0$ for all $k\in\{1,\ldots,d\}$ and, defining $p^{(k)}:=q^{(k)}(X)^{-1}q^{(k)}$, $\big(p^{(1)},\ldots,p^{(d)}\big)$ is a regular finite statistical experiment. We denote by $V_{(X,\mc A)}$ the set of all regular $d$-tuples over $(X,\mc A)$ and by $V$ the class obtained as a union of $V_{(X,\mc A)}$ over all standard Borel measurable spaces $(X,\mc A)$.

When $P=\big(p^{(1)},\ldots,p^{(d)}\big)$ is a regular $d$-tuple, we often use phrases like `$P$-a.a.' or `$\nu\ll P$' for a measure $\nu$ in the sense that we refer to the shared measure class of all the measures $p^{(k)}$. When we need to make a reference to a particular representative of the measure class, we often use $p^{(d)}$.

For $P=\big(p^{(1)},\ldots,p^{(d)}\big)\in V_{(X,\mc A)}$ and $Q=\big(q^{(1)},\ldots,q^{(d)}\big)\in V_{(Y,\mc B)}$, we denote $P\succeq Q$ if there exists a Markov kernel $T(\cdot|\cdot):\mc B\times X\to[0,1]$ such that $q^{(k)}=Tp^{(k)}$, $k=1,\ldots,d$. We also define $P\boxplus Q\in V_{(X\sqcup Y,\mc A\oplus\mc B)}$ and $P\boxtimes Q\in V_{(X\times Y,\mc A\otimes\mc B)}$ through
\begin{align*}
P\boxplus Q&=\big(p^{(1)}\oplus q^{(1)},\ldots,p^{(d)}\oplus q^{(d)}\big),\\
P\boxtimes Q&=\big(p^{(1)}\otimes q^{(1)},\ldots,p^{(d)}\otimes q^{(d)}\big).
\end{align*}
We denote $P\approx Q$ if there is a standard Borel measurable space $(Z,\mc C)$, a regular $d$-tuple $R=\big(r^{(1)},\ldots,r^{(d)}\big)$ over $(Z,\mc C)$, and measurable injections $f:X\to Z$ and $g:Y\to Z$ such that $p^{(k)}=r^{(k)}\circ f^{-1}$ and $q^{(k)}=r^{(k)}\circ g^{-1}$ for $k=1,\ldots,d$. We lift the binary operators and relation $\boxplus$, $\boxtimes$, and $\succeq$ into $V/\!\approx$ denoted respectively by $+$, $\cdot$, $\rgeq$. Given the trivial measurable space $(0,\{\emptyset,\{0\}\})$, we denote the $\approx$-equivalence class of the element $(0,\ldots,0)$ (where $0$ is the zero measure on $(0,\{\emptyset,\{0\})\}$) by $0$ and the $\approx$-equivalence class of the element $(1,\ldots,1)$ (where $1$ is the unique probability measure on $(0,\{\emptyset,\{0\}\})$) by $1$. We may now define the preordered semiring
$$
S^d:=\big(V/\!\approx,0,1,+,\cdot,\rgeq\big).
$$
In the sequel we view functions on $S^d$ as functions on $V$ which are implicitly understood to be constant on $\approx$-equivalence classes.

Next we show that the majorization semiring $S^d$ is amenable to Theorem \ref{thm:Vergleichsstellensatz}, i.e.,\ that $S^d$ is of polynomial growth and possesses the degenerate homomorphism $\|\cdot\|$ of Definition \ref{def:deg}. However, we first prove a very simple but useful lemma that we will use repeatedly in our following proofs. To state this lemma, let us define, for all $\lambda\geq0$, the regular finite statistical experiment
\begin{align}
V_\lambda&=\big(v_\lambda^{(1)},\ldots,v^{(d)}_\lambda\big)\in V_{\big(\{1,\ldots,d\},2^{\{1,\ldots,d\}}\big)},\label{eq:Vlambda}\\
v_\lambda^{(k)}(\ell)&=\frac{1+\lambda\delta_{k,\ell}}{d+\lambda},\qquad k,\ell=1,\ldots,d,
\end{align}
where $\delta_{k,\ell}$ is the Kronecker delta. Here, as often in the sequel, we treat finitely supported measures as functions on the sample set.

\begin{lemma}\label{lemma:apu}
Let $P=\big(p^{(1)},\ldots,p^{(d)}\big)$ be a regular finite statistical experiment over some standard Borel measurable space such that there are $a,b>0$, $a\leq b$, such that $a\leq dp^{(k)}/dp^{(d)}\leq b$ $P$-a.e.. We have $V_\lambda\succeq P$ whenever
\begin{equation}\label{eq:positiivisuus}
\lambda\geq (d-1)\frac{b}{a}+\frac{1}{a}-d.
\end{equation}
\end{lemma}

\begin{proof}
Define $s^\lambda_{k,\ell}:=\big((d+\lambda)\delta_{k,\ell}-1\big)/\lambda$. It is easy to verify that
\begin{equation}\label{eq:kaanteis}
\sum_{\ell=1}^d s^\lambda_{\ell,m}v^{(k)}_\lambda(\ell)=\delta_{k,m},\quad k,\ell=1,\ldots,d.
\end{equation}
If $\sum_{\ell=1}^d s^\lambda_{k,\ell}\,dp^{(\ell)}/dp^{(d)}\geq0$ for $k=1,\ldots,d$, we may define positive measures
$$
T_\lambda(\cdot|k)=\sum_{\ell=1}^d s^\lambda_{k,\ell}p^{(\ell)},\qquad k=1,\ldots,d,
$$
which are immediately seen to be also probability measures. Thus we have a Markov kernel $T_\lambda(\cdot|\cdot)$, and we find
\begin{align*}
T_\lambda v_\lambda^{(k)}&=\sum_{\ell=1}^d T_\lambda(\cdot|\ell)v^{(k)}_\lambda(\ell)=\sum_{\ell,m=1}^d s^\lambda_{\ell,m}v^{(k)}_\lambda(\ell) p^{(m)}=p^{(k)}
\end{align*}
where we have used \eqref{eq:kaanteis} in the final equality. We find that $V_\lambda\succeq P$ as long as $T_\lambda(\cdot|k)$ are positive measures. For this, let us evaluate
\begin{align*}
\frac{dT_\lambda(\cdot|k)}{dp^{(d)}}&=\sum_{\ell=1}^d s^\lambda_{k,\ell}\frac{dp^{(\ell)}}{dp^{(d)}}=\frac{1}{\lambda}\left((d+\lambda)\frac{dp^{(k)}}{dp^{(d)}}-\sum_{\ell=1}^d\frac{dp^{(\ell)}}{dp^{(d)}}\right)\\
&\geq \frac{1}{\lambda}\big((d+\lambda)a-(d-1)b-1\big).
\end{align*}
We immediately see that the final expression is non-negative whenever \eqref{eq:positiivisuus} holds.
\end{proof}

We notice that the above proof can be refined slightly to obtain the following: Let $P=\big(p^{(1)},\ldots,p^{(d)}\big)$ be a regular finite statistical experiment and $a^{(k)},\,b^{(k)}>0$ be such that $a^{(k)}\leq dp^{(k)}/dp^{(d)}\leq b^{(k)}$, $k=1,\ldots,d-1$, $P$-a.e.. Then $V_\lambda\succeq P$ if
$$
\lambda\geq\frac{1}{a^{(k)}}\sum_{\ell=1}^{d-1}b^{(\ell)}+\frac{1}{a^{(k)}}-d
$$
for all $k\in\{1,\ldots,d\}$. This version is useful later in the proof of Proposition \ref{prop:tropical}.

\begin{proposition}\label{lemma:pu}
The preordered semidomain $S^d$ is of polynomial growth and the $\approx$-equivalence class of $U=\big(u^{(1)},\ldots,u^{(d)}\big)\in V_{(Z,\mc C)}$ over some standard Borel $(Z,\mc C)$ is a power universal if and only if $u^{(k)}(Z)=1$ for $k=1,\ldots,d$, (i.e.,\ $U$ is a regular finite statistical experiment) and $u^{(k)}\neq u^{(\ell)}$ whenever $k\neq\ell$.
\end{proposition}

\begin{proof}
Let us first pick a regular finite statistical experiment $P\in V_{(X,\mc A)}$ on a standard Borel measurable space $(X,\mc A)$. Suppose that $R>0$ is such that \eqref{eq:finite} holds for $P$. Substituting $a:=1/R$ and $b=R$, we see that the assumptions of Lemma \ref{lemma:apu} hold, so that, whenever $\lambda\geq(d-1)R^2+R-d.
$ we have $V_\lambda\succeq P$. When we restrict our current semiring to the finite-outcome case, we know that any full-support statistical experiment $V=\big(v^{(1)},\ldots,v^{(d)}\big)$ over a finite sample space, with $k\neq\ell$ $\Rightarrow$ $v^{(k)}\neq v(\ell)$, is a power universal of the restricted semiring; see \cite[Lemma 12]{Farooq_et_al_2024} and \cite{LeJo97}. Especially, for any $\lambda,\mu>0$, $\mu\geq\lambda$, there is $n\in\N$ such that $V_\lambda^{\boxtimes n}\succeq V_\mu$. Thus, for any $\lambda>0$, any standard Borel measurable space, and any regular finite statistical experiment $P=\big(p^{(1)},\ldots,p^{(d)}\big)\in V_{(X,\mc A)}$, using the above result, there is $n\in\N$ such that $V_\lambda^{\boxtimes n}\succeq P$.

Let $U=\big(u^{(1)},\ldots,u^{(d)}\big)\in V_{(Z,\mc C)}$ be a regular finite statistical experiment over a standard Borel measurable space $(Z,\mc C)$ such that $u^{(k)}\neq u^{(\ell)}$ whenever $k\neq\ell$. Whenever $k\neq\ell$, let $C_{k,\ell}\in\mc C$ be such that $u^{(k)}(C_{k,\ell})\neq u^{(\ell)}(C_{k,\ell})$. Let $\mc P=\{D_1,\ldots,D_m\}\subseteq\mc C$ be any finite partition (i.e.,\ $i\neq j$ $\Rightarrow$ $D_i\cap D_j=\emptyset$, $\bigcup_{i=1}^m D_i=Z$) refining the collection of sets $C_{k,\ell}$, i.e.,\ whenever $k\neq\ell$, there is $I_{k,\ell}\subseteq\{1,\ldots,m\}$ such that $C_{k,\ell}=\bigcup_{i\in I_{k,\ell}}D_i$. We clearly have $U\succeq V$ where $V$ is the coarse-grained finite statistical experiment $\big(v^{(1)},\ldots,v^{(d)}\big)$ consisting of the probability distributions $v^{(k)}:\{1,\ldots,d\}\to[0,1]$, $v^{(k)}(i)=u^{(k)}(D_i)$. If there were $k,\ell$, $k\neq\ell$, such that $v^{(k)}=v^{(\ell)}$, we would have $u^{(k)}(D_i)=u^{(\ell)}(D_i)$ for all $i\in\{1,\ldots,m\}$ which would yield the contradictory $u^{(k)}(C_{k,\ell})=u^{(\ell)}(C_{k,\ell})$ upon summing $i\in I_{k,\ell}$. Thus, $v^{(k)}\neq v^{(\ell)}$ whenever $k\neq\ell$. Let us now fix $\lambda>0$. Again, \cite[Lemma 12]{Farooq_et_al_2024} tells us that there is $n_1\in\N$ such that $V^{\boxtimes n_1}\succeq V_\lambda$. Suppose that $(X,\mc A)$ and $(Y,\mc B)$ are standard Borel measurable spaces and $P=\big(p^{(1)},\ldots,p^{(d)}\big)\in V_{(X,\mc A)}$ and $Q=\big(q^{(1)},\ldots,q^{(d)}\big)\in V_{(Y,\mc B)}$ are non-zero regular $d$-tuples such that $P\succeq Q$. Thus, $p^{(k)}(X)=q^{(k)}(Y)=:a_k>0$. Define the regular finite statistical experiment $\tilde{P}=\big(\tilde{p}^{(1)},\ldots,\tilde{p}^{(d)}\big)$ where $\tilde{p}^{(k)}:=a_k^{-1}p^{(k)}$. We view the tuple $(a_1,\ldots,a_d)$ as a regular $d$-tuple over the trivial measurable space $(0,\{\emptyset,\{0\}\})$, so that $Q\succeq(a_1,\ldots,a_d)$ and $(a_1,\ldots,a_d)\boxtimes\tilde{P}=P$. Let $n_2\in\N$ be such that $V_\lambda^{\boxtimes n_2}\succeq\tilde{P}$; this is possible as we have seen above. For $n=n_1\cdot n_2$, we have $U^{\boxtimes n}\succeq V^{\boxtimes n}\succeq V_\lambda^{\boxtimes n_2}\succeq \tilde{P}$, so that
$$
Q\boxtimes U^{\boxtimes n}\succeq(a_1,\ldots,a_d)\boxtimes U^{\boxtimes n}\succeq(a_1,\ldots,a_d)\boxtimes\tilde{P}=P.
$$
This shows that the $\approx$-equivalence class of $U$ is a power universal.

Let $U=\big(u^{(1)},\ldots,u^{(d)}\big)\in V_{(Z,\mc C)}$ be such that the $\approx$-equivalence class of $U$ is a power universal. From the definition of power universals it is obvious that $u^{(k)}(Z)=1$ for $k=1,\ldots,d$. Let $P=\big(p^{(1)},\ldots,p^{(d)}\big)$ be a regular finite statistical experiment such that $p^{(k)}\neq p^{(\ell)}$ whenever $k\neq\ell$ (thus a power universal according to the above). Since $P\succeq (1,\ldots,1)$, where $(1,\ldots,1)$ is the canonical representative of the multiplication unit $1$ of $S^d$, there is $n\in\N$ such that
$$
U^{\boxtimes n}=(1,\ldots,1)\boxtimes U^{\boxtimes n}\succeq P.
$$
If there were non-equal $k$ and $\ell$ such that $u^{(k)}=u^{(\ell)}$, we would now have the contradictory $p^{(k)}=T\big(u^{(k)}\big)^{\otimes n}=T\big(u^{(\ell)}\big)^{\otimes n}=p^{(\ell)}$ for some Markov kernel $T$. Thus, $u^{(k)}\neq u^{(\ell)}$ whenever $k\neq\ell$.
\end{proof}

Let us define $\|\cdot\|:S^d\to\R_{>0}^d\cup\{(0,\ldots,0)\}$ through
$$
\|P\|=\big(p^{(1)}(X),\ldots,p^{(d)}(X)\big)
$$
for all $P=\big(p^{(1)},\ldots,p^{(d)}\big)\in V_{(X,\mc A)}$ over any standard Borel measurable space $(X,\mc A)$. Clearly, $\|\cdot\|$ is a homomorphism. Moreover, $P\succeq Q$ implies $\|P\|=\|Q\|$, i.e.,\ $\|\cdot\|$ is degenerate. Suppose now that $\|P\|=\|Q\|=:(a_1,\ldots,a_d)$ where we view $(a_1,\ldots,a_d)$ as a regular $d$-tuple over the trivial measurable space $(0,\{\emptyset,\{0\}\})$. Now we obviously have
$$
P\succeq (a_1,\ldots,a_d)\preceq Q,
$$
so that the implications
$$
P\succeq Q\ \Rightarrow\ \|P\|=\|Q\|,\quad \|P\|=\|Q\|\ \Rightarrow\ P\sim Q
$$
hold. We have thus established that the prerequisites of Theorem \ref{thm:Vergleichsstellensatz} are met.

\subsection{The monotone homomorphisms and derivations of the majorization semiring}\label{subsec:monhom}

For $\mb K\in\{\R_+,\R_+^{\rm op},\T\R_+,\T\R_+^{\rm op})$, we denote the set of nondegenerate monotone homomorphisms $\Phi:S^d\to\mb K$ by $\Sigma(S^d,\mb K)$. For any $k\in\{1,\ldots,d\}$, we denote the set of derivations associated with $k$'th component of the degenerate homomorphism $\|\cdot\|$ by $\mf D^k(S^d)$. This means that $\Delta\in\mf D^k(S^d)$ is an additive monotone map $S^d\to\R$ satisfying the Leibniz rule:
$$
\Delta(P\boxtimes Q)=\Delta(P)q^{(k)}(Y)+p^{(k)}(X)\Delta(Q)
$$
for all $P=\big(p^{(1)},\ldots,p^{(d)}\big)\in V_{(X,\mc A)}$ and $Q=\big(q^{(1)},\ldots,q^{(d)}\big)\in V_{(Y,\mc B)}$.

In the following result we need the {\it Kullback-Leibler relative entropy} $D_{\rm KL}$ defined for any positive measures $p$ and $q$ (such that $p\ll q$) through
$$
D_{\rm KL}(p\|q)=\int\frac{dp}{dq}\log{\frac{dp}{dq}}\,dq.
$$
The Kullback-Leibler relative entropy satisfies the data processing inequality: if $T$ is a Markov kernel, then $D_{\rm KL}(Tp\|Tq)\leq D_{\rm KL}(p\|q)$.

\begin{proposition}\label{prop:TemperateDeriv}
Suppose that
\begin{itemize}
\item[(i)] $\Phi\in\Sigma(S^d,\R_+)$,
\item[(ii)] $\Phi\in\Sigma(S^d,\R_+^{\rm op})$, or
\item[(iii)] $\Delta\in\mf D^k(S^d)$ for some $k\in\{1,\ldots,d\}$.
\end{itemize}
For any standard Borel $(X,\mc A)$ and any $P=\big(p^{(1)},\ldots,p^{(d)}\big)\in V_{(X,\mc A)}$, we have
\begin{equation}\label{eq:temperateform}
\Phi(P)=\int_X \frac{dp^{(1)}}{dp^{(d)}}(x)^{\alpha_1}\cdots\frac{dp^{(d-1)}}{dp^{(d)}}(x)^{\alpha_{d-1}}\,dp^{(d)}(x)
\end{equation}
in cases (i) and (ii) where $\mathuline{\alpha}=(\alpha_1,\ldots,\alpha_d)\in\R^d$, $\alpha_d:=1-\alpha_1-\cdots-\alpha_{d-1}$, is such that, in case (i), there is $\ell\in\{1,\ldots,d\}$ such that $\alpha_\ell>1$ and $\alpha_m\leq0$ for any $m\neq \ell$ and, in case (ii), $\alpha_\ell\geq0$ for $\ell=1,\ldots,d$, and, in case (iii),
\begin{equation}\label{eq:derivationform}
\Delta(P)=\sum_{\ell:\,\ell\neq k}\gamma_\ell D_{\rm KL}\big(p^{(k)}\big\|p^{(\ell)}\big)
\end{equation}
for some $\gamma_\ell\geq0$.
\end{proposition}

\begin{proof}
Let us start with case (i). We fix $\Phi\in\Sigma(S^d,\R_+)$, a standard Borel $(X,\mc A)$, a non-zero $P=\big(p^{(1)},\ldots,p^{(d)}\big)\in V_{(X,\mc A)}$, and $R>0$ such that $P$ satisfies \eqref{eq:finite}. From our earlier work \cite{Farooq_et_al_2024,Verhagen_et_al_2024}, we know that there is $\mathuline{\alpha}=(\alpha_1,\ldots,\alpha_d)\in\R^d$ such that $\alpha_1+\cdots+\alpha_d=1$ and there is $\ell\in\{1,\ldots,d\}$ such that $\alpha_\ell\geq1$ and $\alpha_m\leq0$ when $m\neq\ell$ so that, when $Y$ is finite and $\mc B=2^Y$ (in which case, with some abuse of notation, we identify measures $\mu:\mc B\to\R_+$ with functions $\mu:Y\to\R_+$) and $Q=\big(q^{(1)},\ldots,q^{(d)}\big)\in V_{(Y,\mc B)}$,
\begin{equation}\label{eq:finitetemperate}
\Phi(Q)=\sum_{y\in Y}q^{(1)}(y)^{\alpha_1}\cdots q^{(d)}(y)^{\alpha_d}.
\end{equation}
We may freely assume that $\alpha_d\geq1$, so that
$$
\Phi(Q)=\sum_{y\in Y}\left(\frac{q^{(1)}(y)}{q^{(d)}(y)}\right)^{\alpha_1}\cdots\left(\frac{q^{(d-1)}(y)}{q^{(d)}(y)}\right)^{\alpha_{d-1}}q^{(d)}(y)
$$
where we tacitly assume that $q^{(\ell)}(y)>0$ for all $y\in Y$ and $\ell\in\{1,\ldots,d\}$.

For any $n\in\N$ sufficiently large (e.g.,\ $n>R$), we consider a finite partition $\{A_{n,i}\}_{i=1}^{m_n}\subseteq\mc A$ of $P$-almost all of $X$ and $a^{(\ell)}_{n,i}\in\R_+$ such that, for any $i\in\{1,\ldots,m_n\}$, $1/(n+1)\leq dp^{(\ell)}/dp^{(d)}-a^{(\ell)}_{n,i}\leq 1/n$ $P$-a.e on $A_{n,i}$. All this is possible (in particular the finiteness of the partition) because of the property \eqref{eq:finite} for $P$. We may thus define the simple functions $f^{(\ell)}_n:X\to\R_+$, $\ell=1,\ldots,d-1$,
$$
f^{(\ell)}_n=\sum_{i=1}^{m_n}a^{(\ell)}_{n,i}\chi_{A_{n,i}}\nearrow \frac{dp^{(\ell)}}{dp^{(d)}}\quad P\textrm{-a.e.}
$$
as $n\to\infty$. Here, as well as in the sequel, for any $A\in\mc A$, $\chi_A$ is the characteristic function of $A$, i.e.,\ $\chi_A(x)=1$ if $x\in A$ and otherwise $\chi_A(x)=0$. Let us define
\begin{align*}
\tilde{P}_n&=\left(\tilde{p}^{(1)}_n,\ldots.\tilde{p}^{(d-1)}_n,\left[1-\frac{1}{n}\right]p^{(d)}\right)\in V_{(X,\mc A)},\\
\tilde{p}^{(\ell)}_n(A)&=\int_A f^{(\ell)}_n\,dp^{(d)},\quad A\in\mc A,\\
R_n&=\left(r^{(1)}_n,\ldots,r^{(d-1)}_n,\frac{1}{n}p^{(d)}\right)\in V_{(X,\mc A)},\\
r^{(\ell)}_n(A)&=\int_A \left(\frac{dp^{(\ell)}}{dp^{(d)}}-f^{(\ell)}_n\right)\,dp^{(d)},\quad A\in\mc A.
\end{align*}
We now easily have, for all $n\in\N$ sufficiently large,
$$
P=\big(\tilde{p}^{(1)}_n+r^{(1)}_n,\ldots,\tilde{p}^{(d-1)}_n+r^{(d-1)}_n,p^{(d)}\big)\preceq \tilde{P}_n\boxplus R_n.
$$

Also, for all $n\in\N$, we define $Y_n:=\{1,\ldots,m_n\}$ and $\mc B_n:=2^{Y_n}$ and $P_n=\big(p^{(1)}_n,\ldots,p^{(d)}_n\big)\in V_{(Y_n,\mc B_n)}$, where $p^{(\ell)}_n(y)=p^{(\ell)}(A_{n,y})$ for all $\ell\in\{1,\ldots,d-1\}$ and $y\in Y_n$. Clearly, $P\succeq P_n$ for any $n\in\N$. This means that, for any $n\in\N$ sufficiently large,
\begin{equation}\label{eq:ekaraj}
P_n\preceq P\preceq \tilde{P}_n\boxplus R_n,
\end{equation}
so that
\begin{equation}\label{eq:Pestimate}
\Phi(P_n)\leq\Phi(P)\leq\Phi(\tilde{P}_n)+\Phi(R_n).
\end{equation}
Our strategy is to show that $\Phi(R_n)\to0$ as $n\to\infty$, so that $\Phi(P)=\lim_{n\to\infty}\Phi(P_n)=\lim_{n\to\infty}\Phi(\tilde{P}_n)$ implying \eqref{eq:temperateform} using the dominated convergence theorem. The same arguments used here largely apply also to cases (ii) and (iii) which is why we treat the latter cases in much less detail.

We may upper-bound the normalized version
$$
\hat{R}_n=\big(r^{(1)}_n(X)^{-1},\ldots,r^{(d-1)}_n(X)^{-1},n p^{(d)}(X)^{-1}\big)\boxtimes R_n=\big(\hat{r}^{(1)}_n,\ldots,\hat{r}_n^{(d)}\big)
$$
of $R_n$ with $V_\lambda$ of \eqref{eq:Vlambda} with some $\lambda\geq0$ which is independent of $n$. Since one may easily verify that
$$
\frac{n}{n+1}\leq\frac{d\hat{r}_n^{(\ell)}}{d\hat{r}_n^{(d)}}\leq\frac{n+1}{n}\quad P\mathrm{-a.e.}
$$
for $\ell=1,\ldots,d-1$, Lemma \ref{lemma:apu} tells us that $V_\lambda\succeq\hat{R}_n$ as long as
$$
\lambda\geq(d-1)\left(\frac{n+1}{n}\right)^2+\frac{n+1}{n}-d.
$$
Thus, considering the maximum value of the right-hand side of the above inequality (at $n=1$), with $\lambda\geq 3d-2$, we have $V_\lambda\succeq\hat{R}_n$ for all $n\in\N$; let us fix such $\lambda$ from now on. Thus, it follows that, denoting
$$
S_n:=\left(r^{(1)}_n(X),\ldots,r^{(d-1)}_n(X),\frac{1}{n}p^{(d)}(X)\right),
$$
which we consider as a tuple over the trivial measurable space, we have
\begin{equation}\label{eq:alayla}
S_n\preceq R_n\preceq S_n\boxtimes V_\lambda
\end{equation}
where the lower bound is obvious. Using the fact that, for all $\ell\in\{1,\ldots,d-1\}$,
$$
\frac{1}{n+1}\leq r^{(\ell)}_n(X)\leq\frac{1}{n}
$$
and recalling that $\alpha_1,\ldots,\alpha_{d-1}\leq0$, we may easily estimate
$$
\frac{1}{n}p^{(d)}(X)^{\alpha_d}\leq \Phi(S_n)\leq \frac{2^{\alpha_d-1}}{n}p^{(d)}(X)^{\alpha_d},
$$
implying that $\Phi(S_n)\to0$ as $n\to\infty$. Using the monotonicity of $\Phi$ and \eqref{eq:alayla}, we find
$$
\underbrace{\Phi(S_n)}_{\to0}\leq\Phi(R_n)\leq\underbrace{\Phi(S_n)}_{\to0}\cdot\underbrace{\Phi(V_\lambda)}_{\rm constant},
$$
showing that $\Phi(R_n)\to0$ as $n\to\infty$.

Let us define, for all $\ell\in\{1,\ldots,d-1\}$ and $n\in\N$, the simple function $g^{(\ell)}_n:X\to\R_+$,
$$
g^{(\ell)}_n=\sum_{i=1}^{m_n}\frac{p^{(\ell)}(A_{n,i})}{p^{(d)}(A_{n,i})}\chi_{A_{n,i}}.
$$
and measure $q^{(\ell)}_n:\mc A\to\R_+$, $q^{(\ell)}_n(A)=\int_A g^{(\ell)}_n\,dp^{(d)}$ for all $A\in\mc A$. Defining $Q_n:=\big(q^{(1)}_n,\ldots,q^{(d-1)}_n,p^{(d)}\big)$ and the measures $p^{(d)}_{A_{n,i}}$ as the restrictions of $p^{(d)}$ onto the restriction of $\mc A$ onto $A_{n,i}$, we observe
\begin{align*}
\Phi(Q_n)&=\Phi\left(\bigboxplus_{i=1}^{m_n}\left(\frac{p^{(1)}(A_{n,i})}{p^{(d)}(A_{n,i})},\ldots,\frac{p^{(d-1)}(A_{n,i})}{p^{(d)}(A_{n,i})},1\right)\boxtimes\big(p^{(d)}_{A_{n,i}},\ldots,p^{(d)}_{A_{n,i}}\big)\right)\\
&=\sum_{i=1}^{m_n}\Phi\left(\frac{p^{(1)}(A_{n,i})}{p^{(d)}(A_{n,i})},\ldots,\frac{p^{(d-1)}(A_{n,i})}{p^{(d)}(A_{n,i})},1\right)\cdot\Phi\big(p^{(d)}_{A_{n,i}},\ldots,p^{(d)}_{A_{n,i}}\big)\\
&=\sum_{i=1}^{m_n}\left(\frac{p^{(1)}(A_{n,i})}{p^{(d)}(A_{n,i})}\right)^{\alpha_1}\cdots\left(\frac{p^{(d-1)}(A_{n,i})}{p^{(d)}(A_{n,i})}\right)^{\alpha_{d-1}}p^{(d)}(A_{n,i})=\Phi(P_n).
\end{align*}
Whenever $x\in A_{n,i}$, we easily find
$$
g^{(\ell)}_n(x)-f^{(\ell)}_n(x)=\frac{1}{p^{(d)}(A_{n,i})}\int_{A_{n,i}}\underbrace{\bigg(\frac{dp^{(\ell)}}{dp^{(d)}}-a^{(\ell)}_{n,i}\bigg)}_{\in[1/(n+1),1/n]}\,dp^{(d)},
$$
implying that, as $f^{(\ell)}_n\nearrow dp^{(\ell)}/dp^{(d)}$ $P$-a.e.\ as $n\to\infty$, also $g^{(\ell)}_n\to dp^{(\ell)}/dp^{(d)}$ $P$-a.e.\ as $n\to\infty$. Thus, also
$$
\left.\begin{array}{r}
f^{(1)}_n(x)^{\alpha_1}\cdots f^{(d-1)}_n(x)^{\alpha_{d-1}}\\
g^{(1)}_n(x)^{\alpha_1}\cdots g^{(d-1)}_n(x)^{\alpha_{d-1}}
\end{array}\right\}\ \to\ \frac{dp^{(1)}}{dp^{(d)}}(x)^{\alpha_1}\cdots\frac{dp^{(d-1)}}{dp^{(d)}}(x)^{\alpha_{d-1}}\quad \textrm{for}\ P\textrm{-a.a.}\ x\in X
$$
as $n\to\infty$. Now, recalling that $\alpha_1,\ldots,\alpha_{d-1}\leq0$ and using the monotone convergence theorem, we have
\begin{align*}
\Phi(\tilde{P}_n)&=\Phi\left(\bigboxplus_{i=1}^{m_n}\left(a^{(1)}_{n,i},\ldots,a^{(d-1)}_{n,i},1-\frac{1}{n}\right)\boxtimes\big(p^{(d)}_{A_{n,i}},\ldots,p^{(d)}_{A_{n,i}}\big)\right)\\
&=\sum_{i=1}^{m_n}\big(a^{(1)}_{n,i}\big)^{\alpha_1}\cdots\big(a^{(d-1)}_{n,i}\big)^{\alpha_{d-1}}\left(1-\frac{1}{n}\right)^{\alpha_d}p^{(d)}(A_{n,i})\\
&=\left(1-\frac{1}{n}\right)^{\alpha_d}\int_X f^{(1)}_n(x)^{\alpha_1}\cdots f^{(d-1)}_n(x)^{\alpha_{d-1}}\,dp^{(d)}(x)\\
&\to\int_X \frac{dp^{(1)}}{dp^{(d)}}(x)^{\alpha_1}\cdots\frac{dp^{(d-1)}}{dp^{(d)}}(x)^{\alpha_{d-1}}\,dp^{(d)}(x)\quad\textrm{as}\ n\to\infty.
\end{align*}
Since the functions $g^{(\ell)}_n$ oscillate with a decreasing bound around $dp^{(\ell)}/dp^{(d)}$, we may prove that
$$
\Phi(P_n)=\Phi(Q_n)\to \int_X \frac{dp^{(1)}}{dp^{(d)}}(x)^{\alpha_1}\cdots\frac{dp^{(d-1)}}{dp^{(d)}}(x)^{\alpha_{d-1}}\,dp^{(d)}(x)\quad\textrm{as}\ n\to\infty
$$
using the dominated convergence theorem. Thus, using \eqref{eq:Pestimate}, we get
$$
\underbrace{\Phi(P_n)-\Phi(\tilde{P}_n)}_{\to0}\leq \Phi(P)-\Phi(\tilde{P}_n)\leq\underbrace{\Phi(R_n)}_{\to0},
$$
so that
\begin{equation}\label{eq:monhomapu}
\Phi(P)=\lim_{n\to\infty}\Phi(\tilde{P}_n)=\int_X \frac{dp^{(1)}}{dp^{(d)}}(x)^{\alpha_1}\cdots\frac{dp^{(d-1)}}{dp^{(d)}}(x)^{\alpha_{d-1}}\,dp^{(d)}(x).
\end{equation}
If $\alpha_d=1$, we would have $\alpha_1=\cdots=\alpha_{d-1}=0$, so that $\Phi(P)=p^{(d)}(X)$ and $\Phi$ would be degenerate. Thus, $\alpha_d>1$.

Let us next show that any map $\Phi$ like that defined in \eqref{eq:monhomapu} with $\alpha_1,\ldots,\alpha_{d-1}\leq0$ is actually a monotone homomorphism in $\Sigma(S^d,\R_+)$; note that, as pointed out earlier, we may freely assume that $\alpha_d$ is the only positive (in fact $>1$) parameter within $\mathuline{\alpha}$. The fact that $\Phi$ is a semiring homomorphism is very easy, so let us concentrate on monotonicity. To aid our investigation, let us define the function $f_{\mathuline{\alpha}}:\R_{>0}\to\R$,
$$
f_{\mathuline{\alpha}}(x_1,\ldots,x_{d-1})=x_1^{\alpha_1}\cdots x_{d-1}^{\alpha_{d-1}}.
$$
Let us show that $f_{\mathuline{\alpha}}$ is convex. The Hessian of $f_{\mathuline{\alpha}}$ is given by
$$
\left(\frac{\partial^2 f_{\mathuline{\alpha}}}{\partial x_k\partial x_\ell}(x_1,\ldots,x_{d-1})\right)_{k,\ell=1}^{d-1}=f_{\mathuline{\alpha}}(x_1,\ldots,x_{d-1})\left(\frac{1}{x_k x_\ell}\right)_{k,\ell=1}^{d-1}\star M
$$
where $\star$ denotes the matrix-element-wise product (the Schur product) of matrices of the same shape and
$$
M:=\big(\alpha_k\alpha_\ell\big)_{k,\ell=1}^{d-1}-{\rm diag}(\alpha_1,\ldots,\alpha_{d-1}).
$$
Since $\alpha_k\leq 0$ for $k=1,\ldots,d-1$, we immediately see that $M$ is positive semi-definite. Since the Schur product of positive semi-definite matrices is also positive semi-definite, we find that the Hessian of $f_{\mathuline{\alpha}}$ is positive semi-definite, so that $f_{\mathuline{\alpha}}$ is convex. We now have that
$$
\Phi(P)=\int_X f_{\mathuline{\alpha}}\left(\frac{dp^{(1)}}{dp^{(d)}},\ldots,\frac{dp^{(d-1)}}{dp^{(d)}}\right)\,dp^{(d)},
$$
i.e.,\ up to a logarithm, $\Phi$ is a multivariate generalization of an $f$-divergence. The standard proof \cite{Csiszar67} of the data processing inequality for $f$-divergences defined by a convex function generalizes now to our case to show the monotonicity of $\Phi$. In the case where $P,Q\in V$, with $P\succeq Q$, do not consist of probability measures, notice that $\|P\|=\|Q\|=:(a_1,\ldots,a_d)$, so that $\hat{P}\succeq\hat{Q}$ for the normalized $\hat{P}:=(a_1^{-1},\ldots,a_d^{-1})\boxtimes P$ and $\hat{Q}:=(a_1^{-1},\ldots,a_d^{-1})\boxtimes Q$ so that
$$
\Phi(P)=f_{\mathuline{\alpha}}\left(\frac{a_1}{a_d},\ldots,\frac{a_{d-1}}{a_d}\right)\Phi(\hat{P})\geq f_{\mathuline{\alpha}}\left(\frac{a_1}{a_d},\ldots,\frac{a_{d-1}}{a_d}\right)\Phi(\hat{Q})=\Phi(Q).
$$

The proof in case (ii) can be done in essentially the same way as above. In this case, we use the fact established in our earlier work \cite{Farooq_et_al_2024,Verhagen_et_al_2024} that the equation \eqref{eq:finitetemperate} still holds for finite tuples, but now all the parameters $\alpha_\ell$ are non-negative; we may naturally assume that $\alpha_d>0$ since $\alpha_1+\cdots+\alpha_d=1$ and proceed as above. The only difference is that the upper and lower bounds in the inequalities are exchanged.

Let us go on to proving the claim in case (iii) and fix $k\in\{1,\ldots,d\}$ and $\Delta\in\mf D^k(S^d)$. We may freely assume that $k=d$. From our earlier work \cite{Farooq_et_al_2024,Verhagen_et_al_2024}, we know that there are $\gamma_\ell\geq0$, out of which at least one $\gamma_\ell$ is non-zero, such that, when $Y$ is finite and $\mc B=2^Y$ (in which case, with some abuse of notation, we identify measures $\mu:\mc B\to\R_+$ with functions $\mu:Y\to\R_+$) and $Q=\big(q^{(1)},\ldots,q^{(d)}\big)\in V_{(Y,\mc B)}$, we have
\begin{align*}
\Delta(Q)&=\sum_{\ell=1}^{d-1}\gamma_\ell D_{\rm KL}\big(q^{(d)}\big\|q^{(\ell)}\big)\\
&=\sum_{\ell=1}^{d-1}\gamma_\ell\sum_{y\in Y} q^{(d)}(y)\log{\frac{q^{(d)}(y)}{q^{(\ell)}(y)}}.
\end{align*}
We may proceed exactly in the same way as in case (i) for $P\big(p^{(1)},\ldots,p^{(d)}\big)\in V_{(X,\mc A)}$ satisfying \eqref{eq:finite} with $R>0$ where \eqref{eq:ekaraj} now implies
\begin{equation}\label{eq:derivestimate}
\Delta(P_n)\leq\Delta(P)\leq\Delta(\tilde{P}_n)+\Delta(R_n)
\end{equation}
and, using the Leibniz rule and fixing $\lambda\geq 3d-2$ as before, we get
$$
\Delta(S_n)\leq\Delta(R_n)\leq\Delta(S_n)+\frac{1}{n}p^{(d)}(X)\Delta(V_\lambda).
$$
Recalling that $1/(n+1)\leq r^{(\ell)}_n(X)\leq 1/n$, we have
$$
\frac{1}{n}\sum_{\ell=1}^{d-1}\gamma_\ell p^{(d)}(X)\log{p^{(d)}(X)}\leq\Delta(S_n)\leq\frac{1}{n}\sum_{\ell=1}^{d-1}\gamma_\ell p^{(d)}(X)\log{\left(\frac{n+1}{n}\cdot p^{(d)}(X)\right)},
$$
implying that $\Delta(S_n)\to0$ as $n\to\infty$, so that also $\Delta(R_n)\to0$ as $n\to\infty$. Using the dominated convergence theorem, we may also prove similarly as in case (i) that
$$
\lim_{n\to\infty}\Delta(P_n)=\lim_{n\to\infty}\Delta(\tilde{P}_n)=\sum_{\ell=1}^{d-1}\gamma_\ell D_{\rm KL}\big(p^{(d)}\big\|p^{(\ell)}\big),
$$
so that, by \eqref{eq:derivestimate}, we have
$$
\Delta(P)=\lim_{n\to\infty}\Delta(\tilde{P}_n)=\sum_{\ell=1}^{d-1}\gamma_\ell D_{\rm KL}\big(p^{(d)}\big\|p^{(\ell)}\big).
$$
On the other hand, all the maps $\Delta$ as defined in \eqref{eq:derivationform} are clearly monotone as the Kullback-Leibler relative entropy satisfies the data processing inequality. Thus, we have completely characterized the derivations in $\mf D^k(S^d)$, $k=1,\ldots,d$.
\end{proof}

From now on, we define, for any $\mathuline{\alpha}=(\alpha_1,\ldots,\alpha_d)\in\R^d$ such that $\alpha_1+\cdots+\alpha_d=1$, the function $\Phi_{\mathuline{\alpha}}:S^d\to\R_+$ through
\begin{equation}\label{eq:Phialpha}
\Phi_{\mathuline{\alpha}}(P)=\int_X \frac{dp^{(1)}}{dp^{(d)}}(x)^{\alpha_1}\cdots\frac{dp^{(d-1)}}{dp^{(d)}}(x)^{\alpha_{d-1}}\,dp^{(d)}(x)
\end{equation}
for all standard Borel $(X,\mc A)$ and $P=\big(p^{(1)},\ldots,p^{(d)}\big)\in V_{(X,\mc A)}$. Proposition \ref{prop:TemperateDeriv} tells us that
\begin{align*}
\Sigma(S^d,\R_+)&=\bigcup_{k=1}^d \{\Phi_{\mathuline{\alpha}}\,|\,\alpha_1+\cdots+\alpha_d=1,\ \alpha_k>1,\ \alpha_\ell\leq0\ \textrm{whenever}\ \ell\neq k\},\\
\Sigma(S^d,\R_+^{\rm op})&=\{\Phi_{\mathuline{\alpha}}\,|\,\alpha_1+\cdots+\alpha_d=1,\ \alpha_1,\ldots,\alpha_d\geq0\}.
\end{align*}
Moreover, for any $k\in\{1,\ldots,d\}$, the set $\mf D^k(S^d)$ of derivations coincides with the positive cone generated by the Kullback-Leibler relative entropies of the $k$'th measure with the others, i.e.,
$$
\mf D^k(S^d)=\textrm{cone}\,\big\{\big(p^{(1)},\ldots,p^{(d)}\big)\mapsto D_{\rm KL}\big(p^{(k)}\big\|p^{(\ell)}\big)\,\big|\,\ell\neq k\big\}.
$$
Indeed, we have seen that derivations within $\mf D^k(S^d)$ are $\Delta^{(k)}_{\mathuline{\gamma}}$ for $\mathuline{\gamma}=(\gamma_1,\ldots,\gamma_d)\in\R_+^d$,
\begin{equation}\label{eq:GammaDeriv}
\Delta^{(k)}_{\mathuline{\gamma}}(P)=\sum_{\ell:\,\ell\neq k}\gamma_\ell D_{\rm KL}\big(p^{(k)}\big\|p^{(\ell)}\big)
\end{equation}
for all $P=\big(p^{(1)},\ldots,p^{(d)}\big)\in V$.

\begin{proposition}\label{prop:tropical}
Suppose that $\Phi\in\Sigma(S^d,\T\R_+)$. There is $\mathuline{\beta}=(\beta_1,\ldots,\beta_d)\in\R^d$ where $\beta_k=1$ for one $k\in\{1,\ldots,d\}$, $\beta_\ell\leq0$ whenever $\ell\neq k$, $\sum_{\ell:\,\ell\neq k}\beta_\ell=-1$ and
\begin{equation}
\Phi(P)=\underset{x\in X}{P{\rm -ess\,sup}}\,\prod_{\ell:\,\ell\neq k}\frac{dp^{(\ell)}}{dp^{(k)}}(x)^{\beta_\ell}
\end{equation}
for all $P=\big(p^{(1)},\ldots,p^{(d)}\big)\in V$. Moreover, there are no nondegenerate monotone homomorphisms $\Phi:S^d\to\T\R_+^{\rm op}$, i.e.,\ $\Sigma(S^d,\T\R_+^{\rm op})=\emptyset$.
\end{proposition}

\begin{proof}
From our earlier work \cite{Farooq_et_al_2024,Verhagen_et_al_2024}, we know that there are $k\in\{1,\ldots,d\}$ and $\beta_\ell\leq0$ for $\ell\neq k$ (and $\beta_k=1$) as in the statement such that, whenever $Y$ is finite and $\mc B=2^Y$, then
$$
\Phi(Q)=\max_{y\in Y}\prod_{\ell:\,\ell\neq k}\left(\frac{q^{(\ell)}(y)}{q^{(k)}(y)}\right)^{\beta_\ell}
$$
for any $Q=\big(q^{(1)},\ldots,q^{(d)}\big)\in V_{(Y,\mc B)}$ where we identify the measures $q^{(\ell)}$ with functions $q^{(\ell)}:Y\to\R_+$. Naturally, we may assume that $k=d$ as we shall do in the sequel.

Let us fix a standard Borel $(X,\mc A)$ and $P=\big(p^{(1)},\ldots,p^{(d)}\big)\in V_{(X,\mc A)}$ for the duration of this proof. Denote $\mathuline{\beta}:=(\beta_1,\ldots,\beta_{d-1})$ and define the function $g^{\mathuline{\beta}}_P:X\to\R_+$,
$$
g^{\mathuline{\beta}}_P(x)=\frac{dp^{(1)}}{dp^{(d)}}(x)^{\beta_1}\cdots\frac{dp^{(d-1)}}{dp^{(d)}}(x)^{\beta_{d-1}}.
$$
Let us denote by $G$ the $P$-essential supremum of $g^{\mathuline{\beta}}_P$ (which we will ultimately show to coincide with $\Phi(P)$). We denote, for each $A\in\mc A$ by $P_A$ the restriction of $P$ onto $A$, i.e.,\ $P_A=\big(p^{(1)}_A,\ldots,p^{(d)}_A\big)$ where $p^{(\ell)}_A:\mc A|_A\to\R_+$ is the restriction of $p^{(\ell)}$ onto the $\sigma$-algebra $\mc A|_A:=\{A\cap B\,|\,B\in\mc A\}$. Thus, we have $P=\bigboxplus_{A\in\mc P}P_A$ for any finite partition $\mc P\subseteq\mc A$. Thus, for any finite partition $\mc P\subseteq\mc A$, we have
$$
\Phi(P)=\max_{A\in\mc P}\Phi(P_A).
$$
Especially, for any $A\in\mc A$,
\begin{equation}
\Phi(P)\geq\Phi(P_A)\geq\Phi\big(p^{(1)}(A),\ldots,p^{(d)}(A)\big)=\left(\frac{p^{(1)}(A)}{p^{(d)}(A)}\right)^{\beta_1}\cdots\left(\frac{p^{(d-1)}(A)}{p^{(d)}(A)}\right)^{\beta_{d-1}}
\end{equation}
where the second inequality follows immediately from the monotonicity of $\Phi$. We next show that
\begin{equation}\label{eq:Gsup}
G=\sup_{A\in\mc A}\left(\frac{p^{(1)}(A)}{p^{(d)}(A)}\right)^{\beta_1}\cdots\left(\frac{p^{(d-1)}(A)}{p^{(d)}(A)}\right)^{\beta_{d-1}},
\end{equation}
so that $\Phi(P)\geq G$.

Let us define the path $\R_+\ni\lambda\mapsto\mathuline{\alpha}^\lambda=\big((\lambda-1)\beta_1,\ldots,(\lambda-1)\beta_{d-1},\lambda\big)\in\R^d$. We notice that, when $0\leq\lambda<1$, $\Phi_{\mathuline{\alpha}^\lambda}\in\Sigma(S^d,\R_+^{\rm op})$ and, when $1<\lambda<\infty$, $\Phi_{\mathuline{\alpha}^\lambda}\in\Sigma(S^d,\R_+)$. For any positive measure $\mu$ on a measurable space $(Y,\mc B)$ and $s\geq1$, we denote the $(\mu,s)$-norm of a $\mu$-measurable function $f:Y\to\R$ by $\|f\|_{\mu,s}:=\left(\int_Y|f|^s\,d\mu\right)^{1/s}$ whenever the integral exists. An application of H\"{o}lder's inequality tells us that, whenever $\mu(Y)<\infty$, $\mu(Y)^{-1/s}\|f\|_{\mu,s}\nearrow \|f\|_{\mu,\infty}:=\mu\textrm{-ess}\,\sup|f|$ as $s\to\infty$ whenever all the norms are defined. We thus have
\begin{align*}
 \lim_{\lambda\to\infty}\left(\frac{\Phi_{\mathuline{\alpha}^\lambda}(P)}{p^{(d)}(X)}\right)^{\frac{1}{\lambda-1}}&=\lim_{\lambda\to\infty}\left[\int_X\left(\frac{dp^{(1)}}{dp^{(d)}}(x)^{\beta_1}\cdots\frac{dp^{(d-1)}}{dp^{(d)}}(x)^{\beta_{d-1}}\right)^{\lambda-1}\,\frac{dp^{(d)}(x)}{p^{(d)}(X)}\right]^{\frac{1}{\lambda-1}}\\
 &=\lim_{s\to\infty}p^{(d)}(X)^{-1/s}\big\|g^{\mathuline{\beta}}_P\big\|_{p^{(d)},s}=\sup_{s>1}p^{(d)}(X)^{-1/s}\big\|g^{\mathuline{\beta}}_P\big\|_{p^{(d)},s}\\
 &=\big\|g^{\mathuline{\beta}}_P\big\|_{p^{(d)},\infty}=G.
\end{align*}
We may thus deduce $G=\sup_{\lambda>2}\big(\Phi_{\mathuline{\alpha}^\lambda}(P)/p^{(d)}(X)\big)^{\frac{1}{\lambda-1}}$.

Whenever $\mc P=\{A_1,\ldots,A_m\}\subseteq\mc A$ is a finite partition of $X$, we denote $Y_{\mc P}:=\{1,\ldots,m\}$ and $\mc B_{\mc P}:=2^{Y_{\mc P}}$. We again view regular tuples on the finite measurable space $(Y_{\mc P},\mc B_{\mc P})$ as $d$-tuples of functions $q^{(\ell)}:Y_{\mc P}\to\R_+$, $\ell=1,\ldots,d$. For any finite partition $\mc P=\{A_i\}_{i\in Y_{\mc P}}\subseteq\mc A$, we define $P|_{\mc P}=\big(p^{(1)}|_{\mc P},\ldots,p^{(d)}|_{\mc P}\big)\in V_{(Y_{\mc P},\mc B_{\mc P})}$ through $p^{(\ell)}|_{\mc P}(i)=p^{(\ell)}(A_i)$ for all $\ell\in\{1,\ldots,d\}$ and $i\in Y_{\mc P}$. When $\mc P=\{A_i\}_{i\in Y_{\mc P}}\subseteq\mc A$ is a finite partition, we may easily show similarly as above that
$$
\sup_{\lambda>2}\left(\frac{\Phi_{\mathuline{\alpha}^\lambda}(P|_{\mc P})}{p^{(d)}(X)}\right)^{\frac{1}{\lambda-1}}=\max_{i\in Y_{\mc P}}\left(\frac{p^{(1)}(A_i)}{p^{(d)}(A_i)}\right)^{\beta_1}\cdots\left(\frac{p^{(d-1)}(A_i)}{p^{(d)}(A_i)}\right)^{\beta_{d-1}}.
$$
Clearly $P\succeq P|_{\mc P}$ for any finite partition $\mc P$, so that $\Phi'(P)\geq\Phi'(P|_{\mc P})$ for any $\Phi'\in\Sigma(S^d,\R_+)$. However, perusing the proof of Proposition \ref{prop:TemperateDeriv}, we notice that, for any $\Phi'\in\Sigma(S^d,\R_+)$, there exists a sequence $(\mc P_n)_{n=1}^\infty$ of finite partitions such that $\Phi'(P)=\lim_{n\to\infty}\Phi'(P|_{\mc P_n})$; e.g.,\ the tuples $P_n$ of the proof of Proposition \ref{prop:TemperateDeriv} provide such a sequence. Thus, $\Phi'(P)=\sup_{\mc P}\Phi'(P|_{\mc P})$ for all $\Phi'\in\Sigma(S^d,\R_+)$ where the supremum runs over all finite partitions. Thus, using our observation above,
\begin{align*}
G&=\sup_{\lambda>2}\left(\frac{\Phi_{\mathuline{\alpha}^\lambda}(P)}{p^{(d)}(X)}\right)^{\frac{1}{\lambda-1}}=\sup_{\lambda>2}\sup_{\mc P}\left(\frac{\Phi_{\mathuline{\alpha}^\lambda}(P|_{\mc P})}{p^{(d)}(X)}\right)^{\frac{1}{\lambda-1}}=\sup_{\mc P}\sup_{\lambda>2}\left(\frac{\Phi_{\mathuline{\alpha}^\lambda}(P|_{\mc P})}{p^{(d)}(X)}\right)^{\frac{1}{\lambda-1}}\\
&=\sup_{\mc P}\max_{i\in Y_{\mc P}}\left(\frac{p^{(1)}(A_i)}{p^{(d)}(A_i)}\right)^{\beta_1}\cdots\left(\frac{p^{(d-1)}(A_i)}{p^{(d)}(A_i)}\right)^{\beta_{d-1}}\\
&=\sup_{A\in\mc A}\left(\frac{p^{(1)}(A)}{p^{(d)}(A)}\right)^{\beta_1}\cdots\left(\frac{p^{(d-1)}(A)}{p^{(d)}(A)}\right)^{\beta_{d-1}}.
\end{align*}
Thus, we have $\Phi(P)\geq G$. We next show that $\Phi(P)\leq G$.

For any $n\in\N$ (sufficiently large), we find a finite partition $\mc P_n=\{A_{n,1},\ldots A_{n,m_n}\}\subseteq\mc A$ and numbers $a^{(\ell)}_{n,i}$, $b^{(\ell)}_{n,i}$, $b^{(\ell)}_{n,i}\geq a^{(\ell)}_{n,i}$, $b^{(\ell)}_{n,i}-a^{(\ell)}_{n,i}\leq 1/n$, such that, for all $\ell\in\{1,\ldots,d-1\}$ and $i\in\{1,\ldots,m_n\}$,
$$
\frac{dp^{(\ell)}}{dp^{(d)}}(x)\in\big[a^{(\ell)}_{n,i},b^{(\ell)}_{n,i}\big]\quad\mathrm{for}\ P\mathrm{-a.a.}\ x\in A_{n,i}.
$$
Let us denote $P_{n,i}=\big(p^{(1)}_{n,i},\ldots,p^{(d)}_{n,i}\big)\in V_{(A_{n,i},\mc A|_{A_{n,i}})}$, $p^{(\ell)}_{n,i}(B)=p^{(\ell)}(A_{n,i})^{-1}p^{(\ell)}(B)$ for any $\mc A$-measurable $B\subseteq A_{n,i}$. With similar analysis as in the proof of Lemma \ref{lemma:apu}, we notice that, whenever
$$
\lambda\geq\frac{b^{(\ell)}}{a^{(\ell)}_{n,i}}\left(\sum_{m=1}^{d-1}\frac{b^{(m)}}{a^{(m)}_{n,i}}+1\right)-d\quad\forall \ell\in\{1,\ldots,d\},
$$
we have $V_\lambda\succeq P_{n,i}$. When $R$ is such that $P$ satisfies \eqref{eq:finite}, we notice that $\lambda_n\geq0$ that satisfies the above condition for all $\ell\in\{1,\ldots,d\}$ and $i\in\{1,\ldots,m_n\}$ is given by
$$
\lambda_n:=\left(1+\frac{R}{n}\right)\left[(d-1)\left(1+\frac{R}{n}\right)+1\right]-d\to0\ \mathrm{as}\ n\to\infty.
$$
We now have, for all $n\in\N$,
$$
P=\bigboxplus_{i=1}^{m_n}\big(p^{(1)}(A_{n,i}),\ldots,p^{(d)}(A_{n,i})\big)\boxtimes P_{n,i}\preceq\bigboxplus_{i=1}^{m_n}\big(p^{(1)}(A_{n,i}),\ldots,p^{(d)}(A_{n,i})\big)\boxtimes V_{\lambda_n}.
$$
Using the monotonicity of $\Phi$, the easily verifiable fact that $\Phi(V_\lambda)=\lambda+1$, and \eqref{eq:Gsup}, we have
\begin{align*}
\Phi(P)&\leq\max_{1\leq i\leq m_n}\Phi\big(p^{(1)}(A_{n,i}),\ldots,p^{(d)}(A_{n,i})\big)\cdot\Phi(V_{\lambda_n})\\
&=(\lambda_n+1)\max_{1\leq i\leq m_n}\underbrace{\left(\frac{p^{(1)}(A_{n,i})}{p^{(d)}(A_{n,i})}\right)^{\beta_1}\cdots\left(\frac{p^{(d-1)}(A_{n,i})}{p^{(d)}(A_{n,i})}\right)^{\beta_{d-1}}}_{\leq G}\leq(\lambda_n+1)G.
\end{align*}
Since this evaluation holds for all $n\in\N$ and $\lambda_n\to0$ as $n\to\infty$, we have $\Phi(P)\leq G$. Thus, we have $\Phi(P)=G$.

We still need to show that any map $\Phi$ having the form of the claim is actually monotone. However, we have essentially already proven that, whenever $P\in V_{(X,\mc A)}$ is non-zero and $\lambda\mapsto\mathuline{\alpha}^\lambda$ is the path appearing earlier in this proof, we have
$$
\Phi(P)=\lim_{\lambda\to\infty}\left(\frac{\Phi_{\mathuline{\alpha}^\lambda}(P)}{p^{(d)}(X)}\right)^{\frac{1}{\lambda-1}}.
$$
As the quantities appearing in the limit are monotone, so is $\Phi$ as the pointwise limit.

Let us now prove the final claim. We pick $\Phi\in\Sigma(S^d,\T\R_+^{\rm op})$. We first show that, whenever $a>0$, $\Phi(a,\ldots,a)=1$; this also holds in the $\T\R_+$-case, but that already follows from our earlier work \cite{Farooq_et_al_2024,Verhagen_et_al_2024}. We first show this for a rational $a>0$. When $n\in\N$, we immediately have $\Phi(n,\ldots,n)=\Phi(1,\ldots,1)=1$ just using the tropical additivity of $\Phi$. Moreover, whenever $m\in\N$,
\begin{align*}
1&=\Phi(1,\ldots,1)=\Phi\left((m,\ldots,m)\boxtimes\left(\frac{1}{m},\ldots,\frac{1}{m}\right)\right)\\
&=\Phi(m,\ldots,m)\cdot\Phi\left(\frac{1}{m},\ldots,\frac{1}{m}\right)=\Phi\left(\frac{1}{m},\ldots,\frac{1}{m}\right).
\end{align*}
Now, when $m,n\in\N$,
$$
\Phi\left(\frac{n}{m},\ldots\frac{n}{m}\right)=\Phi(n,\ldots,n)\cdot\Phi\left(\frac{1}{m},\ldots,\frac{1}{m}\right)=1.
$$
Thus, $\Phi(a)=1$ for any rational $a>0$. We next pick any $a>0$ and some rational $p,q>0$, $p<a<q$. We immediately see
$$
(a,\ldots,a)\succeq(a-p,\ldots,a-p)\boxplus(p,\ldots,p)\succeq(a,\ldots,a),
$$
so that $\Phi(a,\ldots,a)=\max\{\Phi(a-p,\ldots,a-p),\Phi(p,\ldots,p)\}\geq\Phi(p,\ldots,p)=1$. On the other hand,
$$
(q,\ldots,q)\succeq(q-a,\ldots,q-a)\boxplus(a,\ldots,a)\succeq(q,\ldots,q),
$$
so that $1=\Phi(q,\ldots,q)=\max\{\Phi(q-a,\ldots,q-a),\Phi(a,\ldots,a)\}\geq\Phi(a,\ldots,a)$. Putting these observations together, we have $\Phi(a,\ldots,a)=1$.

Let $(X,\mc A)$ be standard Borel and consider $t\in(0,1)$, $U=\big(u^{(1)},\ldots,u^{(d)}\big)\in V_{(X,\mc A)}$, and $V=\big(v^{(1)},\ldots,v^{(d)}\big)\in V_{(X,\mc A)}$. We denote
$$
tU+(1-t)V:=\big(tu^{(1)}+(1-t)v^{(1)},\ldots,tu^{(d)}+(1-t)v^{(d)}\big)\in V_{(X,\mc A)}.
$$
It is obvious that $tU+(1-t)V\preceq tU\boxplus(1-t)V$. Thus, using the observation above together with the monotonicity of $\Phi$, we have
\begin{align*}
\Phi\big(tU+(1-t)V\big)&\geq\Phi\big(tU\boxplus(1-t)V\big)\\
&=\max\{\Phi(U)\cdot\Phi(t,\ldots,t),\Phi(V)\cdot\Phi(1-t,\ldots,1-t)\}\\
&=\max\{\Phi(U),\Phi(V)\}.
\end{align*}
This already looks pretty bad for $\Phi$, but let us be thorough.

Let us fix $P=\big(p^{(1)},\ldots,p^{(d)}\big)\in V_{(X,\mc A)}$ for some standard Borel $(X,\mc A)$. Let us assume that $R>0$ is such that $P$ satisfies \eqref{eq:finite}. Let us denote $P_0:=\big(p^{(d)},\ldots,p^{(d)}\big)$. Whenever $0<t<\min\{1,1/R\}$, we may define $U=\big(u^{(1)},\ldots,u^{(d)}\big)\in V_{(X,\mc A)}$ where $u^{(\ell)}=(1-t)^{-1}\big(p^{(\ell)}-tp^{(d)}\big)$ and $V=\big(v^{(1)},\ldots,v^{(d)}\big)\in V_{(X,\mc A)}$ where $v^{(\ell)}=(1-t)^{-1}\big(p^{(d)}-tp^{(\ell)}\big)$. We thus have $P=tP_0+(1-t)U$ and $P_0=tP+(1-t)V$, so that
\begin{align*}
\Phi(P)&=\Phi\big(tP_0+(1-t)U)\geq\max\{\Phi(P_0),\Phi(U)\}\geq\Phi(P_0)\\
&=\Phi\big(tP+(1-t)V\big)\geq\max\{\Phi(P),\Phi(V)\}\geq\Phi(P).
\end{align*}
Hence, we have $\Phi(P)=\Phi(P_0)=1$ since $P_0\succeq\big(p^{(d)}(X),\ldots,p^{(d)}(X)\big)\succeq P_0$ and $\Phi\big(p^{(d)}(X),\ldots,p^{(d)}(X)\big)=1$. Thus, $\Phi(P)=1$ whenever $P\neq0$ and $\Phi(0)=0$, i.e.,\ $\Phi$ is the trivial tropical homomorphism which is degenerate. This proves the final claim.
\end{proof}

For any $\mathuline{\beta}=(\beta_1,\ldots,\beta_d)\in\R^d\setminus\{0\}$ such that $\beta_1+\cdots+\beta_d=0$ and $\beta_k=1$ for one $k$, let us define the function $\Phi^\T_{\mathuline{\beta}}:S^d\to\R_+$ through
\begin{equation}
\Phi^\T_{\mathuline{\beta}}(P)=\underset{x\in X}{P{\rm -ess\,sup}}\,\prod_{\ell:\,\ell\neq k}\frac{dp^{(\ell)}}{dp^{(k)}}(x)^{\beta_\ell}
\end{equation}
for any $P=\big(p^{(1)},\ldots,p^{(d)}\big)\in V_{(X,\mc A)}$ on any standard Borel $(X,\mc A)$. According to Proposition \ref{prop:tropical}, we have
\begin{align*}
\Sigma(S^d,\T\R_+)&=\bigcup_{k=1}^d\{\Phi^\T_{\mathuline{\beta}}\,|\,\beta_1+\cdots+\beta_d=0,\ \beta_k=1,\ \beta_\ell\leq0\ \mathrm{whenever}\ \ell\neq k\},\\
\Sigma(S^d,\T\R_+^{\rm op}\}&=\emptyset.
\end{align*}

\section{Multivariate R\'{e}nyi divergences and their applications in majorization}\label{sec:MRD}

Let us denote by $\mc P^d$ the section of $S^d$ of those $P\in S^d$ such that $\|P\|=(1,\ldots,1)$, i.e.,\ $\mc P^d$ is the set of regular finite statistical experiments. This is naturally the set of actual information-theoretic interest.

\begin{definition}
We say that a map $D:\mc P^d\to\R$ is a {\it (multivariate) divergence} if
\begin{itemize}
\item[(i)] $P\succeq Q$ $\Rightarrow$ $D(P)\geq D(Q)$ (monotonicity) and
\item[(ii)] $D(P\boxtimes Q)=D(P)+D(Q)$ for all $P,Q\in\mc P^d$ (extensivity).
\end{itemize}
We denote the set of divergences $D:\mc P^d\to\R$ by $\mf D(\mc P^d)$.
\end{definition}

Let us investigate a divergence $D:\mc P^d\to\R$. Extensivity of $D$ immediately gives us $D(1,\ldots,1)=0$. Since, for any probability measure $p$, $(p,\ldots,p)\succeq(1,\ldots,1)\succeq(p,\ldots,p)$, we also have $D(p,\ldots,p)=0$ due to monotonicity. Since $P\succeq(1,\ldots,1)$ for all $P\in\mc P^d$, we have $D(P)\geq D(1,\ldots,1)=0$. Thus, in fact, $D:\mc P^d\to\R_+$.

The consistent way of defining divergences on $\mc P^d$ using the monotone homomorphisms and derivations is using a fixed power universal. Let us consider the fixed power universal $U:=V_1$ where $V_\lambda$, $\lambda>0$, are are as in \eqref{eq:Vlambda}. Recall that, due to Proposition \ref{lemma:pu}, these are power universals. Recalling that $\Sigma(S^d,\T\R_+^{\rm op})$ is empty, we define the sets $\mf D(S^d,\mb K)$, $\mb K\in\{\R_+,\R_+^{\rm op},\T\R_+\}$,
\begin{equation}\label{eq:Div}
\mf D(S^d,\mb K):=\left\{\mc P^d\ni P\mapsto\frac{\log{\Phi(P)}}{\log{\Phi(U)}}\,\middle|\,\Phi\in\Sigma(S^d,\mb K)\right\}
\end{equation}
of divergences. The sets of derivations $\mf D^k(S^d)$ when restricted on $\mc P^d$ are already divergences. For $\Phi_{\mathuline{\alpha}}\in\Sigma(S^d,\R_+)\cup\Sigma(S^d,\R_+^{\rm op})$, $\mathuline{\alpha}=(\alpha_1,\ldots,\alpha_d)$, we have
$$
\Phi_{\mathuline{\alpha}}(U)=\frac{1}{d+1}\sum_{\ell=1}^d 2^{\alpha_\ell}.
$$
As a function of $\mathuline{\alpha}$ this is somewhat cumbersome and there does not seem to be other power universals in $\mc P^d$ that would result into anything prettier. This is why we replace the above normalization with the simpler $\max_{1\leq k\leq d}\alpha_k-1$ which as a function of $\mathuline{\alpha}$ shares the important properties with the function $\mathuline{\alpha}\mapsto\log{\Phi_{\mathuline{\alpha}}(U)}$. This prompts the following definition.

\begin{definition}\label{def:MRD}
Let us denote by $A_+\subset\R^d$ the set of those $\mathuline{\alpha}=(\alpha_1,\ldots,\alpha_d)$ such that $0\leq\alpha_\ell<1$ for $\ell=1,\ldots,d$ and $\alpha_1+\cdots+\alpha_d=1$. We also denote, for any $k\in\{1,\ldots,d\}$, by $A_k$ the set of those $\mathuline{\alpha}=(\alpha_1,\ldots,\alpha_d)$ such that $\alpha_k>1$, $\alpha_\ell\leq0$ whenever $\ell\neq k$, and $\alpha_1+\cdots+\alpha_d=1$. We also define $A_-:=A_1\cup\cdots\cup A_d$. For any $\mathuline{\alpha}=(\alpha_1,\ldots,\alpha_d)\in A_+\cup A_-$, we define the {\it temperate multivariate R\'{e}nyi divergence} $D_{\mathuline{\alpha}}:\mc P^d\to\R_+$ through
\begin{align}
D_{\mathuline{\alpha}}(P)&=\frac{1}{\max_{1\leq \ell\leq d}\alpha_\ell-1}\log{\Phi_{\mathuline{\alpha}}(P)}\nonumber\\
&=\frac{1}{\max_{1\leq \ell\leq d}\alpha_\ell-1}\log{\int_X\frac{dp^{(1)}}{dp^{(d)}}(x)^{\alpha_1}\cdots\frac{dp^{(d-1)}}{dp^{(d)}}(x)^{\alpha_{d-1}}\,dp^{(d)}(x)}\label{eq:TemperateDiv}
\end{align}
for any $P=\big(p^{(1)},\ldots,p^{(d)}\big)\in\mc P^d$ on any standard Borel $(X,\mc A)$.

For any $k\in\{1,\ldots,d\}$, let us denote by $B_k$ the set of those $\mathuline{\beta}=(\beta_1,\ldots,\beta_d)\in\R^d$ such that $\beta_k=1$, $\beta_\ell\leq0$ whenever $\ell\neq k$, and $\beta_1+\cdots+\beta_d=0$. We also denote $B:=B_1\cup\cdots\cup B_d$. For any $\mathuline{\beta}=(\beta_1,\ldots,\beta_d)\in B$, we define the {\it tropical multivariate R\'{e}nyi divergence} $D^\T_{\mathuline{\beta}}:\mc P^d\to\R_+$ through
\begin{align}\label{eq:TropicalDiv}
D^\T_{\mathuline{\beta}}(P)&=\log{\Phi^\T_{\mathuline{\beta}}(P)}=\log{\left\{\underset{x\in X}{P{\rm -ess\,sup}}\,\prod_{\ell:\,\ell\neq k}\frac{dp^{(\ell)}}{dp^{(k)}}(x)^{\beta_\ell}\right\}}
\end{align}
for any $P=\big(p^{(1)},\ldots,p^{(d)}\big)\in\mc P^d$ on any standard Borel $(X,\mc A)$.

We denote the set of all of the above temperate and tropical divergences $D_{\mathuline{\alpha}},D^\T_{\mathuline{\beta}}:\mc P^d\to\R$ together with the derivations $\Delta^{(k)}_{\mathuline{\gamma}}$, $k=1,\ldots,d$, $\mathuline{\gamma}=(\gamma_1,\ldots,\gamma_d)\in\R_+^d$, $\sum_{\ell:\,\ell\neq k}\gamma_\ell=1$, of \eqref{eq:GammaDeriv} by $\hat{\mf D}(\mc P^d)$. See also Figure \ref{fig:Restricted} for a visualization in case $d=3$.
\end{definition}

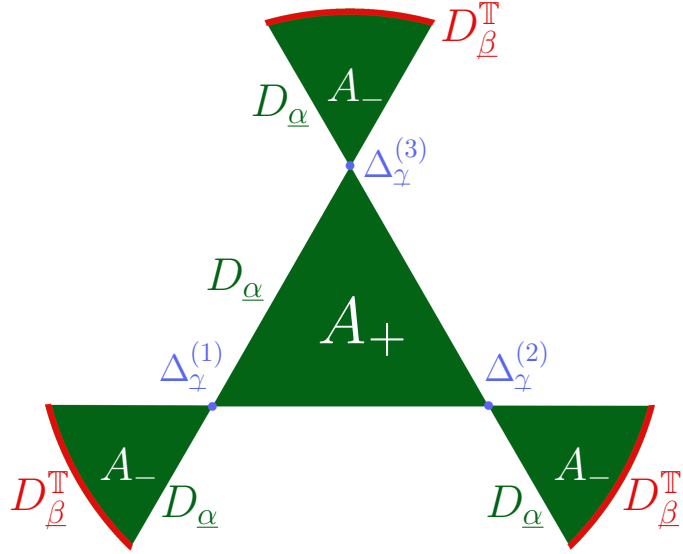
\begin{figure}
\begin{center}
\begin{overpic}[scale=0.45,unit=1mm]{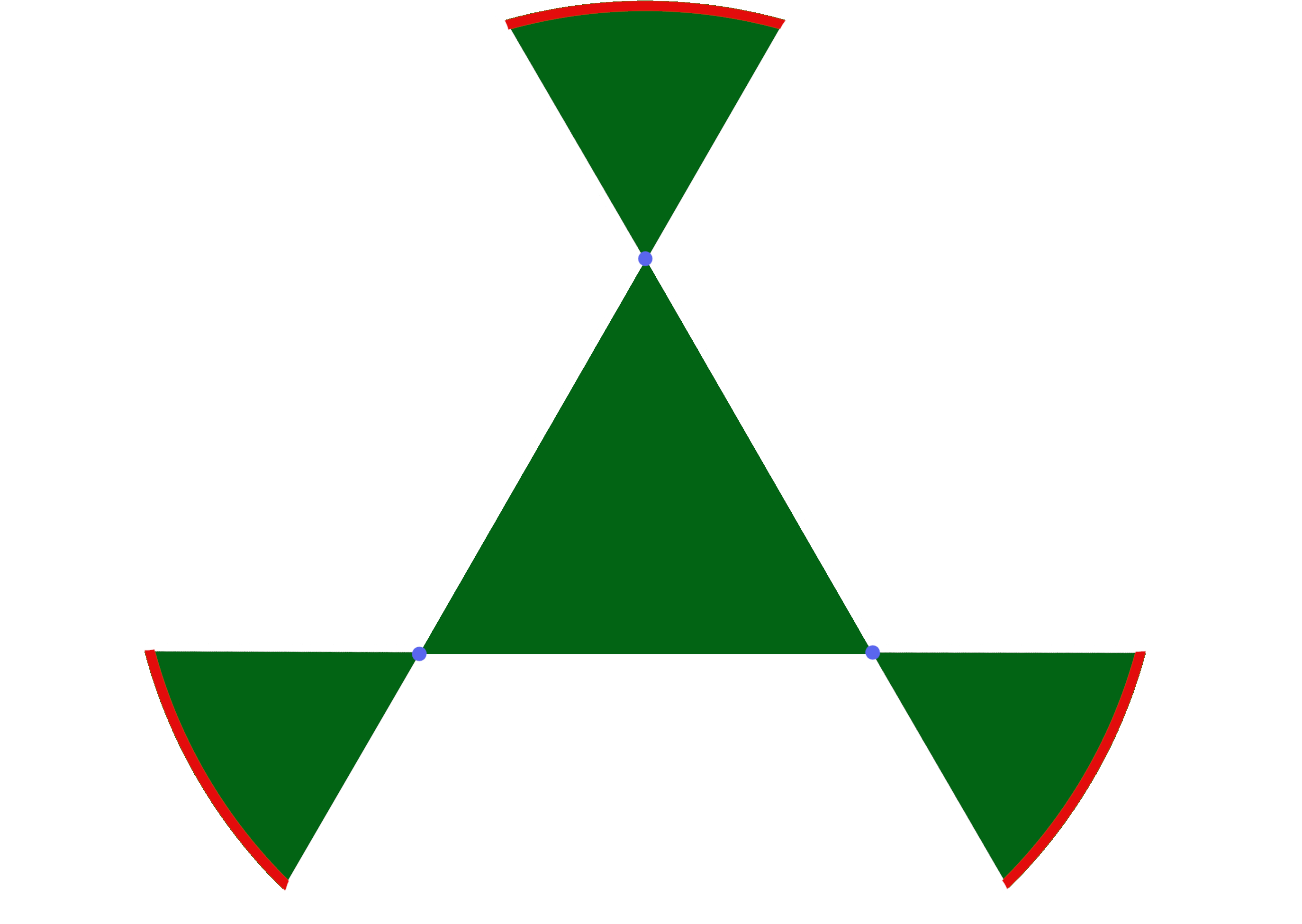}
\put(45,30){\begin{Huge}
${\color{white} A_+}$
\end{Huge}}
\put(17,12){\begin{Large}
${\color{white} A_-}$
\end{Large}}
\put(77,12){\begin{Large}
${\color{white} A_-}$
\end{Large}}
\put(47,62){\begin{Large}
${\color{white} A_-}$
\end{Large}}
\put(31,37){\begin{Large}
\darker{$D_{\mathuline{\alpha}}$}
\end{Large}}
\put(25,7){\begin{Large}
\darker{$D_{\mathuline{\alpha}}$}
\end{Large}}
\put(68,7){\begin{Large}
\darker{$D_{\mathuline{\alpha}}$}
\end{Large}}
\put(37,60){\begin{Large}
\darker{$D_{\mathuline{\alpha}}$}
\end{Large}}
\put(25,25){\begin{large}
\sininen{$\Delta^{(1)}_{\mathuline{\gamma}}$}
\end{large}}
\put(68,25){\begin{large}
\sininen{$\Delta^{(2)}_{\mathuline{\gamma}}$}
\end{large}}
\put(52,52){\begin{large}
\sininen{$\Delta^{(3)}_{\mathuline{\gamma}}$}
\end{large}}
\put(5,8){\begin{Large}
\punainen{$D^\T_{\mathuline{\beta}}$}
\end{Large}}
\put(86,8){\begin{Large}
\punainen{$D^\T_{\mathuline{\beta}}$}
\end{Large}}
\put(62,70){\begin{Large}
\punainen{$D^\T_{\mathuline{\beta}}$}
\end{Large}}
\end{overpic}
\caption{\label{fig:Restricted} The multivariate divergences of Definition \ref{def:MRD} together with the derivations $\Delta^{(k)}_{\mathuline{\gamma}}$ of \eqref{eq:GammaDeriv} depicted in the case $d=3$. The figure is contained in the affine plane of $\R^3$ of those points whose coordinates sum up to 1. The temperate divergences are associated with the areas $A_+$ and $A_-$ and are depicted in green. Note that, e.g.,\ in the lower-left part of $A_-$, we have $\alpha_1>1$ and $\alpha_2,\alpha_3\leq0$. The red arcs represent points at infinity and they correspond to the tropical divergences. The blue points are $e_1=(1,0,0)$, $e_2=(0,1,0)$, and $e_3=(0,0,1)$. The situation depicted in this figure generalizes to the general $d\in\N$ case in a straightforward manner: the parameter area $A_+$ corresponds to the $d$-dimensional probability simplex (without its extreme points) and $A_-$ extends from the extreme points of the probability simplex as cones reaching to infinity.}
\end{center}
\end{figure}

Let us fix $k\in\{1,\ldots,d\}$ and $\mathuline{\beta}=(\beta_1,\ldots,\beta_d)\in B_k$ and define the path
$$
\R_+\ni\lambda\mapsto\mathuline{\alpha}^\lambda=\big((\lambda-1)\beta_1,\ldots,(\lambda-1)\beta_{k-1},\lambda,(\lambda-1)\beta_{k+1},\ldots,(\lambda-1)\beta_d\big)\in\R^d.
$$
We immediately see that, when $0\leq\lambda<1$, $\mathuline{\alpha}^\lambda\in A_+$ and, when $\lambda>1$, $\mathuline{\alpha}^\lambda\in A_k$. Moreover, denoting the smallest of $\beta_\ell$, $\ell\neq k$ by $\beta_{\rm min}$ (so that $\beta_{\rm min}$ has the largest absolute value of all $\beta_\ell$, $\ell\neq k$), we find that $\lambda$ is the largest entry in $\mathuline{\alpha}^\lambda$ if and only if
$$
\lambda\geq\lambda_0:=\frac{\beta_{\rm min}}{\beta_{\rm min}-1}.
$$
Using Jensen's inequality, it is easy to see that
$$
\lambda_0\leq\lambda_1\leq\lambda_2,\ \lambda_1\neq 1\neq\lambda_2\quad\Longrightarrow\quad D_{\mathuline{\alpha}^{\lambda_1}}(P)\leq D_{\mathuline{\alpha}^{\lambda_2}}(P).
$$
Moreover, it is easy to verify that, for any $P=\big(p^{(1)},\ldots,p^{(d)}\big)\in\mc P^d$,
\begin{align*}
\lim_{\lambda\to 1}D_{\mathuline{\alpha}^\lambda}(P)&=-\sum_{\ell:\,\ell\neq k}\beta_\ell D_{\rm KL}\big(p^{(k)}\big\|p^{(\ell)}\big)=\Delta^{(k)}_{-\mathuline{\beta}}(P),\\
\lim_{\lambda\to\infty}D_{\mathuline{\alpha}^\lambda}(P)&=\sup_{\lambda>1}D_{\mathuline{\alpha}^\lambda}(P)=D^\T_{\mathuline{\beta}}(P).
\end{align*}
This shows that we obtain all the derivations and all the tropical divergences as pointwise limits of the temperate divergences.

\subsection{Sufficient conditions for large-sample and catalytic majorization}

With the slight deviation of using a different normalization from the one used in Subsection \ref{subsec:preorderedsemirings}, the temperate and tropical multivariate R\'{e}nyi divergences of Definition \ref{def:MRD} and mixtures of the binary Kullback-Leibler divergences form the test spectrum $\hat{\mf D}(S^d)$ of the majorization semiring $S^d$ which we may identify with the set $\hat{\mf D}(\mc P^d)$ of Definition \ref{def:MRD}. A straightforward application of Theorem \ref{thm:Vergleichsstellensatz} gives us the following result on large-sample and catalytic majorization generalizing the bivariate ($d=2$) result of \cite{Mu_et_al_2021} and the multivariate results on finite sample spaces of \cite{Farooq_et_al_2024}. Note that the experiments $P$ and $Q$ of the following theorem are assumed to be elements of the semiring $S^d$, so condition \eqref{eq:finite} is assumed to holds for both with some $R>0$. Furthermore, conditions in \eqref{eq:derivcond} together with the fact that $D_{\rm KL}(p\|q)>0$ $\Rightarrow$ $p\neq q$ imply that $P=\big(p^{(1)},\ldots,p^{(d)}\big)$ is such that $p^{(k)}\neq p^{(\ell)}$ whenever $k\neq\ell$ so that, according to Proposition \ref{lemma:pu}, $P$ is a power universal.

\begin{theorem}\label{thm:LSCatExact}
Let $P=\big(p^{(1)},\ldots,p^{(d)}\big)\in\mc P^d$ 
and $Q=\big(q^{(1)},\ldots,q^{(d)}\big)\in\mc P^d$. If
\begin{align}
D_{\mathuline{\alpha}}(P)&>D_{\mathuline{\alpha}}(Q),\hspace{62pt}\forall\mathuline{\alpha}\in A_+\cup A_-,\label{eq:tempcond}\\
D^\T_{\mathuline{\beta}}(P)&>D^\T_{\mathuline{\beta}}(Q),\hspace{62pt}\forall\mathuline{\beta}\in B,\label{eq:tropcond}\\
D_{\rm KL}\big(p^{(k)}\big\|p^{(\ell)}\big)&>D_{\rm KL}\big(q^{(k)}\big\|q^{(\ell)}\big),\qquad\forall k,\ell\in\{1,\ldots,d\},\ k\neq\ell,\label{eq:derivcond}
\end{align}
then $P^{\boxtimes n}\succeq Q^{\boxtimes n}$ for $n\in\N$ sufficiently large and there also exists $R\in\mc P^d$ such that $P\boxtimes R\succeq Q\boxtimes R$. In the final part, we may always choose the catalyst $R=\big(r^{(1)},\ldots,r^{(d)}\big)$ according to
$$
r^{(k)}=\frac{1}{n+1}\bigoplus_{m=0}^n \big(p^{(k)}\big)^{\otimes(n-m)}\big(q^{(k)}\big)^{\otimes m},\qquad k=1,\ldots,d,
$$
for $n\in\N$ sufficiently large. On the other hand, if $P^{\boxtimes n}\succeq Q^{\boxtimes n}$ for some $n\in\N$ or $P\boxtimes R\succeq Q\boxtimes R$ for some $R\in\mc P^d$, then the inequalities in \eqref{eq:tempcond}, \eqref{eq:tropcond}, and \eqref{eq:derivcond} hold non-strictly.
\end{theorem}

\subsection{Barycentric representations of multivariate divergences}

We next characterize all multivariate divergences $D\in\mf D(\mc P^d)$ following the proof of Theorem 2 of \cite{Mu_et_al_2021}. The cornerstone of the proof is the following simple lemma:

\begin{lemma}\label{lemma:step}
The implication
$$
\forall\Delta\in\hat{\mf D}(\mc P^d):\ \Delta(P)\geq\Delta(Q)\quad\Longrightarrow\quad D(P)\geq D(Q)
$$
holds for any multivariate divergence $D\in\mf D(\mc P^d)$.
\end{lemma}

\begin{proof}
Let us fix $D\in\mf D(\mc P^d)$. Let us first assume that $P,Q\in\mc P^d$ are such that $\Delta(P)>\Delta(Q)$ for all $\Delta\in\hat{\mf D}(\mc P^d)$. Using Theorem \ref{thm:Vergleichsstellensatz} and fixing a power universal $U\in\mc P^d$ (see Proposition \ref{lemma:pu}), we find $n,k\in\N$ such that $P^{\boxtimes n}\boxtimes U^{\boxtimes k}\succeq Q^{\boxtimes n}\boxtimes U^{\boxtimes k}$. Thus, we have
$$
nD(P)+kD(U)=D\big(P^{\boxtimes n}\boxtimes U^{\boxtimes k}\big)\geq D\big(Q^{\boxtimes n}\boxtimes U^{\boxtimes k}\big)=nD(Q)+kD(U),
$$
so that $D(P)\geq D(Q)$. Let us next assume that $P,Q\in\mc P^d$ are such that $\Delta(P)\geq\Delta(Q)$ for all $\Delta\in\hat{\mf D}(\mc P^d)$. We make these inequalities strict by using powers of $P$ and $Q$ and the power universal $U$: for all $n\in\N$
$$
\Delta\big(P^{\boxtimes n}\boxtimes U\big)=n\Delta(P)+\underbrace{\Delta(U)}_{>0}>n\Delta(P)\geq n\Delta(Q)=\Delta\big(Q^{\boxtimes n}\big)
$$
for all $\Delta\in\hat{\mf D}(\mc P^d)$. Thus, according to the first part of this proof, we have
$$
nD(P)+D(U)=D\big(P^{\boxtimes n}\boxtimes U\big)\geq D\big(Q^{\boxtimes n}\big)=nD(Q)
$$
for all $n\in\N$, i.e.,\ $D(P)+D(U)/n\geq D(Q)$ for all $n\in\N$. Thus, $D(P)\geq D(Q)$.
\end{proof}

Using the above lemma in the same way as in the proof of Theorem 2 of \cite{Mu_et_al_2021} (see the proof of Theorem 7 of \cite{haapasalo2025} for details), we may show that any $D\in\mf D(\mc P^d)$ can be expressed as a barycentre of a finite positive measure on the compactification of the set $A_+\cup A_-$. This compactification is to be viewed as $\hat{\mf D}(\mc P^d)$ which is a compact Hausdorff space in the topology of pointwise convergence \cite{FritzII}. This means that, instead of a single point included at the limits $\mathuline{\alpha}\to e_k$ (a natural basis vector of $\R^d$), we have the full $(d-1)$-dimensional probability simplex corresponding to the derivations $\Delta^{(k)}_{\mathuline{\gamma}}$ where $\mathuline{\gamma}=(\gamma_1,\ldots,\gamma_d)\in\R_+^d$, $\sum_{\ell:\,\ell\neq k}\gamma_\ell=1$. Naturally, these simplexes are spanned by the $d-1$ extreme points giving rise to (possible) point measure contributions from the pairwise Kullback-Leibler relative entropies. The pointwise limits at the infinity coincide with the tropical divergences $D^\T_{\mathuline{\beta}}$, $\mathuline{\beta}\in B$. In this setting, using Lemma \ref{lemma:step}, we may define a positive linear functional $H$ on the subspace of the Banach space of continuous functions on the compact Hausdorff space $\hat{\mf D}(\mc P^d)$ spanned by the evaluation functions ${\rm ev}_P$ at a statistical experiment $P$, ${\rm ev}_P(\Delta)=\Delta(P)$ for all $\Delta\in\hat{\mf D}(\mc P^d)$, through $H({\rm ev}_P)=D(P)$. Using a theorem due to Kantorovich \cite{Kantorovich} related to the Hahn-Banach theorem, we may extend $H$ into a positive linear functional $I$ defined on the whole space of continuous functions on $\hat{\mf D}(\mc P^d)$. According to the Riesz-Markov-Kakutani theorem, there is a(n inner and outer regular) finite measure $\mu:\mc B\big(\hat{\mf D}(\mc P^d)\big)\to\R_+$ (where the Borel $\sigma$-algebra $\mc B\big(\hat{\mf D}(\mc P^d)\big)$ is defined w.r.t.\ the topology of pointwise convergence) such that $I(f)=\int f\,d\mu$. This means that we arrive at the following:

\begin{theorem}\label{thm:barycentre}
We shall denote by $\mc B(C)$ the Borel (or Lebesgue) $\sigma$-algebra of any measurable subset $C$ of $\R^m$ for any $m\in\N$. Suppose that $D\in\mf D(S^d)$ is a multivariate divergence. There are finite measures $\mu:\mc B(A_+\cup A_-)\to\R_+$ and $\nu:\mc B(B)\to\R_+$ and $\gamma_{k,\ell}\geq0$, $k,\ell=1,\ldots,d$ such that
$$
D(P)=\int_{A_+\cup A_-}D_{\mathuline{\alpha}}(P)\,d\mu(\mathuline{\alpha})+\int_B D^\T_{\mathuline{\beta}}(P)\,d\nu(\mathuline{\beta})+\sum_{k,\ell=1}^d\gamma_{k,\ell}D_{\rm KL}\big(p^{(k)}\big\|p^{(\ell)}\big)
$$
for all $P=\big(p^{(1)},\ldots,p^{(d)}\big)\in\mc P^d$.
\end{theorem}

\begin{remark}
Let us highlight that this result has also been previously independently proven in \cite{balsubramani2026}, therein as Theorem 5.1. This goes to show that same results can be reached with greatly different methods. The biggest difference between the approaches taken in this paper and in \cite{balsubramani2026} is that, in this paper, Theorem \ref{thm:barycentre} is a consequence of Theorem \ref{thm:LSCatExact} proven using real-algebraic techniques and the results of \cite{haapasalo2025} whereas, in \cite{balsubramani2026}, the methodology similar with \cite{haapasalo2025} is used, but the preceding proof techniques are different \cite{balsubramani2026a}.
\end{remark}

\subsection{Optimal achievable rates}\label{subsec:rates}

We now derive the form of optimal transformation rate between finite statistical experiments. However, we start in the more general setting of Subsection \ref{subsec:preorderedsemirings}. We let $S$ be a zerosumfree preordered semidomain of polynomial growth and degeneracy $d$ for some $d\in\N$. Fixing a power universal $u\in S$, we define the test spectrum $\hat{\mf D}(S)$. Let us consider $x,y\in S$, $x\sim y$, where $x$ is a power universal, so that $x\sim 1$. Thus, we have $1\sim x\sim y$, i.e.,\ $\|x\|=\|y\|=(1,\ldots,1)$. We say that $r\geq0$ is {\it $(x,y)$-achievable} if there is a sequence $(m_n)_{n=1}^\infty\in\N^\N$ of natural numbers such that $r\leq\liminf_{n\to\infty}m_n/n$ and $x^n\rgeq y^{m_n}$ for all $n\in\N$ sufficiently large. The {\it optimal achievable rate of transformation} $x\to y$ is given by
$$
r(x\to y):=\sup\{r\geq0\,|\, r\ \mathrm{is}\ (x,y)\mathrm{-achievable}\}.
$$
The following result is Corollary 29 of \cite{Verhagen_et_al_2024}.

\begin{corollary}\label{cor:rates}
Let $x,y\in S$, $x\sim y$, where $x$ is a power universal. We have
\begin{equation}\label{eq:ratemin}
r(x\to y)=\min_{\Delta\in\hat{\mf D}(S)}\frac{\Delta(x)}{\Delta(y)}.
\end{equation}
\end{corollary}

Let us now concentrate on multivariate majorization. Given $P,Q\in\mc P^d$, the optimal rate $r(P\to Q)$ corresponds to the optimal ratio of transforming copies of $P$ into copies of $Q$ with a single stochastic map, i.e.,\ a Markov kernel. Recalling that the total set $\hat{\mf D}(\mc P^d)$ (which corresponds to the test spectrum $\hat{\mf D}(S^d)$) of multivariate R\'{e}nyi divergences is the closure (w.r.t.\ the topology of pointwise convergence) of the set of $D_{\mathuline{\alpha}}$ with $\mathuline{\alpha}\in A_+\cup A_-$, we immediately obtain the following result as a consequence of Corollary \ref{cor:rates}.

\begin{corollary}
Suppose that $P,Q\in\mc P^d$ and that the probability measures $p^{(k)}$ appearing in $P$ are pairwise distinct. The optimal transformation rate of copies of $P$ into copies of $Q$ is given by
$$
r(P\to Q)=\min_{\Delta\in\hat{\mf D}(\mc P^d)}\frac{\Delta(P)}{\Delta(Q)}=\inf_{\mathuline{\alpha}\in A_+\cup A_-}\frac{D_{\mathuline{\alpha}}(P)}{D_{\mathuline{\alpha}}(Q)}.
$$
\end{corollary}

\section{Conclusions}\label{sec:concl}

We have derived sufficient and almost necessary conditions for large-sample and catalytic majorization of finite statistical experiments using real-algebraic methods, highlighting the power of the theory of asymptotic spectra also in the non-discrete setting of the current work. The multivariate R\'{e}nyi divergences characterizing large-sample majorization also determine all the general multivariate divergences through barycentres. Thus, we have generalized the work done in the continuous bivariate case \cite{Mu_et_al_2021} and the results on multivariate matrix majorization \cite{Farooq_et_al_2024}.

In this work, we have concentrated on the case where all the probability measures within the same statistical experiment are mutually dominating, i.e.,\ belong to the same measure class. However, it is also possible to study more general support conditions in the same way as in \cite{Verhagen_et_al_2024}. Let us highlight two special cases of support conditions within the same statistical experiment $P=\big(p^{(1)},\ldots,p^{(d)}\big)$ over some standard Borel measurable space $(X,\mc A)$:
\begin{itemize}
\item For all $k\in\{1,\ldots,d\}$, $p^{(k)}=p^{(k)}_\parallel+p^{(k)}_\perp$ where the two measures are orthogonal and $p^{(k)}_\parallel\ll p^{(\ell)}_\parallel$ for all $k,\ell\in\{1,\ldots,d\}$. This case corresponds to the minimal restrictions semiring studied in \cite{Verhagen_et_al_2024}.
\item For all $k\in\{1,\ldots,d-1\}$, $p^{(k)}\ll p^{(d)}$. This case corresponds to the dominating column semiring of \cite{Verhagen_et_al_2024}. In this case, it seems likely that we have to require
$$
\underset{x\in X}{p^{(d)}\mathrm{-ess\,sup}}\,\frac{dp^{(k)}}{dp^{(d)}}(x)<\infty,\qquad k=1,\ldots,d-1,
$$
taken into account the results of \cite{Verhagen_et_al_2024} and the methods that we have used in this work.
\end{itemize}
In \cite{Verhagen_et_al_2024}, the power universals were exhaustively characterized in the above two cases (and were conjectured in a variety of other cases). We anticipate that similar characterizations are possible also in the continuous case and they should lead us to proper continuous generalizations of the characterization of large-sample matrix majorization with varying support conditions.

Our main result, Theorem \ref{thm:LSCatExact} requires the boundedness condition of \eqref{eq:finite} to hold both for the input experiment $P$ and the output experiment $Q$. Since this condition is quite restrictive, it would be desirable to have a less strict condition. Let us point out that the condition of \eqref{eq:finite} for $P$ simply means that the tropical divergences $D^\T_{\mathuline{\beta}}(P)$ are well defined for $\mathuline{\beta}=(\beta_1,\ldots,\beta_d)\in B$ such that $\beta_k=1$, $\beta_\ell=-1$, and $\beta_m=0$ otherwise. This immediately means that all the tropical divergences are well defined. In other words, the condition in \eqref{eq:finite} just means that the tropical divergences are well defined and we need this property to state the conditions \eqref{eq:tropcond} among the sufficient conditions for large-sample and catalytic majorization in Theorem \ref{thm:LSCatExact}. To lift the boundedness condition in \eqref{eq:finite}, it thus seems that we have to discuss a form of majorization where the tropical conditions do not play a role. In vector majorization, we know that the counterpart of the temperate conditions are enough to characterize catalytic majorization \cite{klimesh2007inequalities} whereas the counterparts of tropical conditions are required when studying large-sample majorization \cite[Corollary 46]{Farooq_et_al_2024}. Thus, it might be possible to drop the condition \eqref{eq:finite} or at least replace it with something less stringent if we only want to characterize catalytic multivariate majorization. Another possibility is studying asymptotic majorization in stead. Asymptotic large-sample majorization of a statistical experiment $P=\big(p^{(1)},\ldots,p^{(d)}\big)$ over $Q=\big(q^{(1)},\ldots,q^{(d)}\big)$ means that, for any $\varepsilon>0$, there exists a statistical experiment $Q_\varepsilon=\big(q^{(1)}_\varepsilon,\ldots,q^{(d)}_\varepsilon\big)$ such that $\big\|q^{(k)}-q^{(k)}_\varepsilon\big\|_{\rm TV}<\varepsilon$ for $k=1,\ldots,d$ (i.e.,\ $Q_\varepsilon$ is $\varepsilon$-close to $Q$ in total variation distance) such that $P$ majorizes $Q_\varepsilon$ in large samples, i.e.,\ $P^{\boxtimes n}\succeq Q^{\boxtimes n}_\varepsilon$ for any $n\in\N$ large enough. As the tropical divergences are pointwise limits of the temperate quantities, it seems likely that some form of asymptotic majorization only requires non-strict versions of the conditions \eqref{eq:tempcond}, suggesting that something weaker than \eqref{eq:finite} is required. However, we leave this for future investigation.

\section*{Acknowledgements}

The author would like to thank Frits Verhagen, Marco Tomamichel, Mil\'{a}n Mosonyi, and P\'{e}ter Vrana for stimulating discussions. The author acknowledges support from the National Research Foundation Investigatorship Award (NRF-NRFI10-2024-0006) and from the National Research Foundation, Singapore through the National Quantum Office, hosted in A*STAR, under its Centre for Quantum Technologies Funding Initiative (S24Q2d0009)

\bibliographystyle{ultimate}
\bibliography{bibliography}

\end{document}